\documentclass[12pt]{article}

\usepackage{graphicx}
\usepackage{amssymb,amsmath}
\usepackage[margin=0.75in]{geometry}
\usepackage{mathrsfs}
\usepackage[all]{xy}
\usepackage{hyperref}
\usepackage[usenames,dvipsnames,svgnames,table]{xcolor}

\newtheorem{theorem}{Theorem}[subsection]

\newtheorem{conjecture}{Conjecture}

\newtheorem{definition}[theorem]{Definition}
\newtheorem{example}[theorem]{Example}

\newtheorem{lemma}[theorem]{Lemma}

\newtheorem{proposition}[theorem]{Proposition}
\newtheorem{remark}{Remark}

\numberwithin{equation}{section}

\def\qed{\hfill {\hbox{${\vcenter{\vbox{             
   \hrule height 0.4pt\hbox{\vrule width 0.4pt height 6pt
   \kern5pt\vrule width 0.4pt}\hrule height 0.4pt}}}$}}}

\newenvironment{proof}[1][Proof]{\smallskip\noindent{\bf #1.}\quad}
{\qed\par\medskip}

\title{Column expansion identities and quadratic spanning forest identities}
\author{Melanie Fraser and Karen Yeats\footnote{KY is supported by an NSERC Discovery grant and the Canada Research Chair program.}}

\begin{document}
\maketitle

\begin{abstract}
    Column expansion identities of determinants give a source of quadratic spanning forest polynomial identities and allow us determine the dimension of the space of certain quadratic spanning forest identities, settling a conjecture of the second author with Vlasev from 2012.  Furthermore, we give a combinatorial interpretation of such spanning forest identities via an edge-swapping argument previously developed by the first author in 2019.  Quadratic spanning forest polynomials identities are of particular interest because they are useful for quantum field theory calculations in four dimensions.
\end{abstract}

\section{Introduction}

\subsection{Motivation}

The classical Dodgson identity for a square matrix $M$ with at least 2 rows and columns is 
\[
\det M \det M_{\{1,2\},\{1,2\}} = \det M_{\{1\}, \{1\}} \det M_{\{2\}, \{2\}} - \det M_{\{1\}, \{2\}} \det M_{\{2\}, \{1\}}
\]
where $M_{I,J}$ is the matrix $M$ with the rows indexed by $I$ removed and the columns indexed by $J$ removed.  This formula was made popular by Dodgson in his condensation algorithm \cite{Dodgson}.

When the matrix $M$ is not just any matrix but is a matrix determined by a graph, such as the Laplacian matrix of a graph or other closely related graph matrices, then the Dodgson identity can be interpreted as an identity of spanning forest polynomials of graphs.  This has been observed from a few different directions.  

From the point of view of algebraic geometry and quantum field theory, Francis Brown \cite{Brbig} interpreted the Dodgson identity in terms of Dodgson polynomials -- minors of a version of the graph Laplacian with variables marking the contributions of the different edges, which the second author with Brown \cite{brown} subsequently interpreted as sums of signed spanning forest polynomials.

From the point of view of pure combinatorics and algorithms, the first author, in \cite{fraser}, provided a combinatorial proof of the Dodgson identity by interpreting the Dodgson identity as a quadratic spanning forest identity through the application of the generalized matrix tree theorem.

In both of these cases, in interpreting the classical Dodgson identity, the corresponding spanning forest identities are based off of 3 marked vertices in the graph.  Furthermore, the Dodgson identity is quadratic in the determinants and hence also quadratic in spanning forest polynomials in this interpretation.  However, the usual determinant identities generalizing the classical Dodgson identity are of higher degree.  Naturally interpreting these generalized determinant identities graph theoretically, one obtains spanning forest identities which are also of higher degree in the spanning forest polynomials (see \cite{Brbig, dennis}), while also being based off more than 3 marked vertices.

The quantum field theoretical motivation for studying such determinant and spanning forest identities comes from parametric Feynman integration of some integrals known as Feynman periods in a scalar field theory in 4 dimensions (see \cite{Brbig, Sphi4}).  The expressions for the denominators when integrating one edge at a time are quadratic in the Dodgson or spanning forest polynomials, but often involve more than three marked vertices.  

Consequently, Vlasev along with the second author was interested in finding quadratic 4-vertex spanning forest identities.  In \cite{vlasev} we found the most general such identity possible.  The identity itself is laid out in Section~\ref{subsec 4 vert}.  This identity deals with four marked vertices, and was discovered through a computer program and proved through a somewhat blind manipulation of non-quadratic identities. 

In this paper we will present a more natural and combinatorial derivation of Vlasev and the second author's 4-vertex spanning forest identity, and will generalize to quadratic spanning forest identities on $m$-marked vertices. Note that in Vlasev and the second author's original paper, they denoted the number of marked vertices by $n$. Because most graph theorists understand $n$ to be the total number of vertices in a graph, we will refer to the number of marked vertices as $m$ instead. We hope that this will clarify that the number of marked vertices can be less than the total number of vertices.

As we move up to more marked vertices, there are more possible ways to partition the marked vertices between the trees of the forest.  The classical Dodgson identity, when interpreted in terms of spanning forest polynomials, gives an identity in which the left hand side consists of pairs of a spanning tree and a spanning forest with each marked vertex in a different tree. The right hand side consists of particular pairs of spanning forests with the three marked vertices split between two trees. The 4-vertex identity of \cite{vlasev} relates certain pairs of spanning forests where the left hand side of the identity consists of pairs of a spanning tree and a spanning forest with each marked vertex in a different tree. The right hand side of the identity consists of certain pairs of spanning forests, one of which partitions the 4 marked vertices between two trees and the second of which partitions the 4 marked vertices between three trees. In our generalization we will consider the $m$-vertex identities relating pairs of spanning forests where the left hand side of the identity consists of pairs of a spanning tree and a spanning forest with each of the $m$ marked vertices in a different tree. The right hand side consists of certain pairs of spanning forests, one of which partitions the $m$ marked vertices between two trees and the second of which partitions the $m$ marked vertices between $m-1$ trees.

Vlasev and the second author's 4-vertex identity involves $8$ free variables, and they conjectured at the end of their paper that an $m$-vertex identity of the type outlined above would involve $m(m-2)$ free variables: 

\begin{conjecture}\label{VYConj}
The formulae for quadratic spanning forest identities of the type outlined above and described rigorously in Conjecture~\ref{conj rephrased} on $m$ marked vertices have $m(m-2)$ free variables.
\end{conjecture}

This paper is organized as follows: in the remainder of section 1 we first briefly describe the quantum field theory motivation for quadratic spanning forest identities and then will give necessary notation for the remainder of the paper.  The first of these subsections can be skipped by the reader who is not interested in physics motivation, but subsection~\ref{subsec set up} should not be skipped. In section 2, we will state and prove a set of column expansion identities, and provide a combinatorial interpretation of them. In section 3, we will rephrase Vlasev and the second author's conjecture, use the column expansion identities to build our quadratic spanning forest identities, and from there prove the conjecture.

\subsection{Quantum field theory motivation}
Broadly, quantum field theory is the study of quantum interacting properties.  In perturbative quantum field theory one studies particle interactions by series expansions in some small parameter, often the \emph{coupling} associated with the interaction vertices.  One important family of such expansions are Feynman diagram expansions where the expansion is indexed by certain graphs known as \emph{Feynman diagrams}.  Each Feynman diagram contributes an integral to the expansion, known as the \emph{Feynman integral}.

Computing Feynman integrals is important for high energy physics calculations, for instance of scattering processes at CERN.  Depending on the techniques applied, computing Feynman integrals can have substantial combinatorial aspects.  Francis Brown \cite{Brbig} initiated an approach for integrating Feynman integrals in parametric form, followed up and extended by others such as \cite{Phyp, Bmpl}.  The key ideas of this approach can be seen in the example to which it was initially applied: computing the period of suitably nice scalar Feynman diagrams.  The period is an important residue of the Feynman integral.

To sketch the approach briefly, given a graph $G$ define the Kirchhoff polynomial $\Psi_{G}=\sum_{T}\prod_{e\not\in T}a_e$ where the sum is over spanning trees of $G$.  Then define the period to be
\[
P_G = \int_{a_e\geq 0} \frac{da_2\cdots da_{|E|}}{\Psi_G^2|_{a_1=1}}.
\]
This is but one of many equivalent forms, see \cite{Sphi4}, and converges for sufficiently nice graphs.  Integrating this expression one edge at a time lets us consider the form of the numerator and denominator at each step leaving the substitution $a_1=1$ to the end.  After integrating one edge, say $a_2$, the integrand is  $1/\Psi_{G\backslash 2}\Psi_{G/2}$.  After integrating a second edge the numerator involves logarithms of remaining variables while the denominator is $\Psi_{G\backslash 23}\Psi_{G/23} - \Psi_{G\backslash 2/3}\Psi_{G\backslash 3/2}$, which is amenable to applying the Dodgson identity resulting in a polynomial which is a square of a sum of spanning forest polynomials of the type described in the next section.  The next step can also be explicitly defined in terms of spanning forest polynomials, see \cite{brown}, and the following one, the numerators moving from logarithms to dilogarithms to trilogarithms.  After that the algorithm only continues when the denominator factors.

The details are not important for the present purposes, but what is important is that these denominators are quadratic expressions in spanning forest polynomials, and that identities of quadratic expressions in spanning forest polynomials  generalizing the Dodgson identity are useful for simplifying them and hence better understanding the behaviour of this algorithm.

\subsection{Set up and notation}\label{subsec set up}

The spanning forest polynomials we consider are of the following form.  Given a graph $G$ and a set partition $P$ of a subset of the vertices of $G$, associate a variable $a_e$ to each edge $e$ of $G$.  Then the spanning forest polynomial associated to $G$ and $P$ is 
\[
    \sum_{F\sim P} \prod_{e\not\in F}a_e 
\]
where the sum is over spanning forests $F$ of $G$ where there is a bijection between the trees of the forest and the parts of $P$, such that each vertex in $P$ is in the corresponding tree of $F$.  Note that isolated vertices are allowed as components in our spanning forests.

\begin{figure}\label{triangle}
    \centering
    \includegraphics{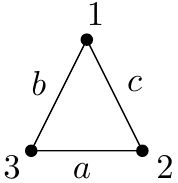}
    \caption{Triangle graph.}
    \label{fig triangle}
\end{figure}

\begin{example}\label{eg triangle}
For example, consider the triangle graph labelled as illustrated in Figure~\ref{fig triangle}.  The spanning forest polynomial associated to the vertex partition $\{1\}, \{2,3\}$ is $bc$ since edges $b$ and $c$ connect vertices in different parts of the partition and so cannot be in any spanning forests compatible with this set partition, while edge $a$ must be the spanning forest as it is the only remaining way to connect vertices $2$ and $3$.  As a second example, the spanning forest polynomial associated to $\{2\}, \{3\}$ is $a(b+c)$ because edge $a$ cannot be in any spanning forest compatible with this partition, while exactly one of edge $b$ or $c$ must be in the spanning forest in order to neither isolate vertex 1 nor connect vertices 2 and 3.
\end{example}

Note that spanning forest polynomials are linear in each variable by definition or to say this another way, every monomial making up the polynomial is squarefree. 

We will provide notation for these spanning forest polynomials that is suited to our needs.  At any given time we will be considering a fixed graph and a fixed set of marked vertices in that graph.  The spanning forests of interest will be spanning forests associated to that graph and to set partitions of the set of marked vertices, or a subset of the set of marked vertices.
This notation follows that of Vlasev and the second author \cite{vlasev}, and we will elaborate in further details on the 4-vertex case, since that is the case which appears there and which serves as a prototype for us.

Fix a graph $G$.
\begin{definition}\label{def partition}
Let $v_1,v_2,\ldots ,v_m$ be $m$ distinct marked vertices. A set partition of a subset of $\{v_1, v_2, \ldots, v_m\}$ will be denoted $(p_1,p_2,\ldots,p_m)$ where $p_i\in\{1,2,3,\dots,m,-\}$. If $p_i=-$ then $v_i$ is not in the subset being partitioned. If $p_i\neq -$, then $v_i$ belongs to the part $p_i$.

In an abuse of notation, $(p_1, p_2,\ldots, p_m)$ also denotes the spanning forest polynomial associated to $G$ and the set partition $(p_1, p_2, \ldots, p_m)$.  
\end{definition}
Since we use such partitions exclusively to index spanning forest polynomials, this notational conflation of the index and the object which is indexed will cause no confusion and will in fact be very handy. 

\begin{example}
Continuing Example~\ref{eg triangle}, the two spanning forest polynomials mentioned explicitly for the triangle graph would be written $(1,2,2)$ and $(-,1,2)$ if we take all three vertices to be marked and take them in the order given by their labels in Figure~\ref{fig triangle}.
\end{example}

\begin{example}
Suppose we have the partition $(1,1,2,-)$. Then vertices $1$ and $2$ both belong to the same part, which is a distinct part from vertex $3$. Vertex $4$ is not in the subset being partitioned. From the perspective of spanning forests, this indicates there are two trees, one containing vertices $1$ and $2$, and one containing vertex $3$.  Vertex $4$ can be in either tree.
\end{example}

There are a few further things to note about Definition~\ref{def partition}.  First, our set partitions do not have ordered parts, so $(1,1,2,-)$ and $(2,2,1,-)$ indicate the same set partition.  Second, if a vertex is not in the subset being partitioned then in the spanning forest corresponding to the set partition, that vertex can appear in any of the trees.  This implies that, for example,
\[
(1,1,2,-) = (1,1,2,1) + (1,1,2,2)
\]
(as spanning forest polynomials).  In general, with $m$ marked vertices 
\[
(p_1, p_2, \ldots, p_{i-1}, -,p_{i+1},  \ldots, p_m) = \sum_{j\in\{p_1,p_2,\ldots,p_m\}} (p_1, p_2, \ldots, p_{i-1}, j, p_{i+1}, \ldots, p_m)
\]
for the same reason, where the sum runs over the distinct values taken on by the $p_j$.

Vlasev and the second author gave names to specific partitions with 4 marked vertices to more concisely write their identity, which we will reproduce here. Other than $(1,1,1,1)$ and $(1,2,3,4)$, the partitions can be grouped into a set in which each partition has three parts (which we will label with $A$s) and a set in which each partition has two parts (which we will label with $B$s). They are

$$\begin{array}{ccc}
    A_1=(1,1,2,3) & A_2=(1,2,1,3) & A_3=(1,2,2,3) \\
    A_4=(1,2,3,1) & A_5=(1,2,3,2) & A_6=(1,2,3,3)
\end{array}
$$

$$\begin{array}{ccc}
    B_1=(1,1,1,2) & B_2=(1,1,2,1) & B_3=(1,2,1,1) \\
    B_4=(1,2,2,2) & B_5=(1,1,2,2) & B_6=(1,2,1,2) \\
    & B_7=(1,2,2,1) &
\end{array}
$$

For $m$ marked vertices, we will likewise write $\{A_i\}$ for the partitions with $m-1$ parts, and $\{B_j\}$ for the partitions with $2$ parts.  However, we will not fix an indexing of these two classes of partitions, but will in later sections notate an $A$ partition by indicating the two indices which form the part of size 2.

In addition to partition notation, we will also need some matrix notation, since our quadratic forest identities will be derived from column expansion identities, which involve determinants of matrices.

\begin{definition}
Let $U$ and $W$ be sets of integers of the same size, and let $M$ be a matrix. Then $M_{U,W}$ is the submatrix of $M$ with the rows corresponding to $U$ removed and the columns corresponding to $W$ removed.
\end{definition}

\begin{example}
The matrix $M_{\{1,2\},\{1,3\}}$ is the matrix $M$ with rows $1$ and $2$ removed and columns $1$ and $3$ removed. As a shorthand, we sometimes drop the set notation in the subscripts when the context is clear. Our shorthand for this example would be $M_{12,13}$.
\end{example}

For the purposes of this paper, we will be looking at undirected graphs. The generalized matrix tree theorem that we will be using to move between determinantal identities and spanning forest identities uses directed graphs in which the rows removed from the Laplacian indicate the roots of trees \cite{chaiken}. This can be easily adapted to undirected graphs by focusing on which vertices are grouped together in a tree instead of where that tree is rooted.

\section{Column Expansion Identities}

We will be using a set of column expansion identities to find and prove our quadratic spanning forest identities. These identities can be thought of as expanding along a specific column in a $k\times k$ cofactor matrix. As such, they come in groups of $k$ identities, one for each possible column to expand along. We will go into more detail later in this section about the connection between these column expansion identities and the quadratic forest identities. For now, let us look at the column expansion identities in their own right.

We will first consider the case with $4$ marked vertices before proceeding to the general case where we will give more formal statements and proofs.  The column expansion identities will lead to quadratic spanning forest polynomial identities of the type of interest to us.  In finding quadratic spanning forest identities on $4$ marked vertices specifically (to compare with Vlasev and the second author), we will need to look at a set of $3$ column expansion identities: 
  
\begin{enumerate}
    \item $\det(M)\det(M_{123,123})=\det(M_{1,1})\det(M_{23,23})-\det(M_{2,1})\det(M_{13,23})+\det(M_{3,1})\det(M_{12,23})$
    \item $\det(M)\det(M_{123,123})=-\det(M_{1,2})\det(M_{23,13})+\det(M_{2,2})\det(M_{13,13})-\det(M_{3,2})\det(M_{12,13})$
    \item $\det(M)\det(M_{123,123})=\det(M_{1,3})\det(M_{23,12})-\det(M_{2,3})\det(M_{13,12})+\det(M_{3,3})\det(M_{12,12})$
\end{enumerate}

Here, $M$ is any square matrix with at least three rows.  
The first identity in the list corresponds to expanding along the first column, the second to expanding along the second column, and the final identity to expanding along the third column.  These identities are perhaps not as well-known as they should be and were pointed out to us by Peter Doyle.

\subsection{Column expansion identities for general $k$}

We will now give a generalization that works for any $k$ where, when $k=3$, we recover the three identities from the previous section.  When we later interpret these in terms of spanning forest polynomials, we will have that the number of marked vertices is $m=k+1$.  Here, $[k]=\{1,2,\dots,k\}$. We will begin with a few definitions.

\begin{definition}
Let $S_k$ be the set of permutations on $[k]$. We will use an adjusted version of $S_{k-1}$ to be the set of bijections from $[k]\setminus\{i\}$ to $[k]\setminus\{j\}$ for fixed elements $i$ and $j$. If $\sigma'\in S_{k-1}$, then we define $\sigma\in S_k$ to be the extension of $\sigma'$ by $\sigma(i)=j$.
\end{definition}

\begin{definition}
We will also use an inversion vector for each permutation $\tau$. The $i$th element of the inversion vector for $\tau$ is the number of elements greater than $i$ to the left of $i$ in $\tau$. The number of inversions in $\tau$ (denoted $\iota(\tau)$) is the sum of elements in its inversion vector.
\end{definition}

Armed with these definitions, we can now state our first lemma.

\begin{lemma}\label{invsigns}
For a fixed $i$ and $j$, given $\sigma'\in S_{k-1}$ and its extension $\sigma\in S_k$ as defined above, $$(-1)^{\iota(\sigma)-(i+j)}=(-1)^{\iota(\sigma')}$$
\end{lemma}

\begin{proof}
Consider our inversion vector for $\sigma$. In order to get the inversion vector for $\sigma'$, we are removing the $i$th element of the inversion vector for $\sigma$. Let there be $\ell$ elements smaller than $i$ to the right of $i$ in $\sigma$. Then $i$ contributes $1$ to each of those $\ell$ elements in the inversion vector for $\sigma$. When we remove $i$ to get to $\sigma'$, each of those $\ell$ elements in the inversion vector for $\sigma'$ will be smaller by $1$. We also want to know what the $i$th element of the inversion vector for $\sigma$ is. There are $i-\ell-1$ elements smaller than $i$ to the left of $i$, and there are $j-1$ total elements to the left of $i$ (since by definition $\sigma(i)=j$), so there are $j-1-(i-\ell-1)=j-i+\ell$ elements greater than $i$ to the left of $i$. Thus $$\iota(\sigma')=\iota(\sigma)-\ell-(j-i+\ell)=\iota(\sigma)-2\ell-(j-i).$$
What we actually care about is the sign of these permutations, so we have $$(-1)^{\iota(\sigma')}=(-1)^{\iota(\sigma)-2\ell-(j-i)}=(-1)^{\iota(\sigma)-(i+j)}.$$
\end{proof}

We can now state the column expansion identities. Theorem \ref{columnexpansion} can be derived from
Theorem 1.7 of \cite{Williamson}, but we will give a proof that is
self-contained and leads into the combinatorial arguments of section 2.3. An even more direct proof can be given by expanding the second determinant on the right hand side using row expansion along row $i$ and then applying the identity $M \cdot \text{adj } M = \det M
\cdot I$. We thank Darij Grinberg for
pointing this out to us. 

\begin{theorem}\label{columnexpansion} \textbf{Column Expansion Identities.}
For any square matrix $M$ and any integer $k$, there are $k$ column expansion identities, one for each $j$ where $1\leq j\leq k$, of the form
$$\det(M)\det(M_{[k],[k]})=\sum_{i=1}^k(-1)^{i+j}\det(M_{i,j})\det(M_{[k]\setminus\{i\},[k]\setminus\{j\}}).$$
\end{theorem}

\begin{proof}
We will begin with an algebraic proof that derives the column expansion identities from the Dodgson-Muir identity \cite{brualdi1983}. The Dodgson-Muir identity states:

\begin{equation*}
    \det(M)\det(M_{[k],[k]})^{k-1}=\sum_{\sigma\in S_k}(-1)^{\iota(\sigma)}\prod_{i=1}^k \det(M_{[k]\setminus \{i\},[k]\setminus \{\sigma(i)\}})
\end{equation*}

Fix $j$ such that $1\leq j\leq k$. This will give us the column expansion identity expanding along column $j$. We begin by factoring $\det(M_{[k]\setminus\{i\},[k]\setminus\{j\}})$ out of the Dodgson-Muir identity for all $i$: 
\begin{align*}
\det(M)&\det(M_{[k],[k]})^{k-1}=\\&\sum_{i=1}^k (-1)^{i+j}\det(M_{[k]\setminus\{i\},[k]\setminus\{j\}})\left(\sum_{\sigma\in S_k, \sigma(i)=j}(-1)^{\iota(\sigma)-(i+j)}\prod_{l=1,l\neq i}^k\det(M_{[k]\setminus\{l\},[k]\setminus\{\sigma(l)\}})\right).
\end{align*}

We would now like to simplify what is in the parentheses. Using our definitions of $S_{k-1}$ and $\sigma'$ from above, the inside of the parentheses becomes: $$\sum_{\sigma'\in S_{k-1}}(-1)^{\iota(\sigma)-(i+j)}\prod_{l\in[k]\setminus \{i\}}\det(M_{[k]\setminus\{l\},[k]\setminus\{\sigma'(l)\}}).$$

 Let $M'=M_{i,j}$, and let $[k-1]$ be $[k]\setminus\{i\}$ or $[k]\setminus \{j\}$ depending on the context. Using this notation and the results of Lemma \ref{invsigns}, we can further simplify the inside of the parentheses to: $$\sum_{\sigma'\in S_{k-1}}(-1)^{\iota(\sigma')}\prod_{l\in[k-1]}\det(M'_{[k-1]\setminus\{l\},[k-1]\setminus\{\sigma'(l)\}}).$$
Then this is the right hand side of Dodgson-Muir, so it equals $\det(M')\det(M'_{[k-1],[k-1]})^{k-2}$. Plugging this into the parentheses and replacing $M'$ with $M_{i,j}$ gives us:
$$\det(M)\det(M_{[k],[k]})^{k-1}=\sum_{i=1}^k(-1)^{i+j}\det(M_{[k]\setminus\{i\},[k]\setminus\{j\}})(\det(M_{i,j})\det(M_{[k],[k]})^{k-2}).$$
Viewing the entries of the matrices as indeterminants and hence the determinants as polynomials in those indeterminants, we can then divide both sides by $\det(M_{[k],[k]})^{k-2}$, which gives us our result.

\end{proof}

\subsection{Combinatorial Interpretation of Column Expansion Identities}\label{sec comb int}

We will use the all minors matrix tree theorem \cite{chaiken} to derive quadratic spanning forest identities from the column expansion identity. The all minors matrix tree theorem relies on the Laplacian of a graph, defined below.

Suppose we have a directed graph with a variable or weight assigned to each directed edge.  For an undirected graph, take each edge to be a pair of directed edges, one in each direction, with the same associated weight.  Set the weight to be $0$ for non-edges.
\begin{definition}
Let $a_{ij}$ be the weight of the edge $i\to j$. Define the \textbf{Laplacian} $A$ by 
\[A_{ij}= \begin{cases} 
      -a_{ij} & i\neq j \\
      \displaystyle\sum_{m\neq i}a_{im} & i=j
   \end{cases}
\]
\end{definition}

\begin{remark}
The explicit statement of the all minors matrix tree theorem involves several sign-based functions that become irrelevant in our particular context, so we will forgo stating it here. For an explicit statement of the all minors matrix tree theorem, see \cite{chaiken}. In the context of this paper, we focus on minors of the Laplacian, which represent signed forests. In each tree of the forest, there is exactly one vertex from the set of removed rows and exactly one vertex from the set of removed columns. Therefore the size of the set used to make the Laplacian minor corresponds to the number of trees in the forest. For a particular Laplacian minor, we will get a sum of signed forests that all satisfy the vertex condition above. The sign of the forest corresponds to the sign it contributes to the determinant of the Laplacian. To find the sign of a forest, we can think of it in terms of a permutation array, where each entry in the permutation array corresponds to an edge in the forest. More details on the sign are discussed below.
\end{remark}

In order to obtain our spanning forest identities, we will replace $M$ in the column expansion identities with the Laplacian of a complete graph ($L$) with a row and column already removed. Because we need to remove a row and column for the matrix tree theorem to work, when we are looking for quadratic spanning forest identities on $m$ special vertices, we can imagine that $M$ in the column expansion identities is a matrix with a row and column already removed from the Laplacian. Thus, we will look at the column expansion identities associated with $k=m-1$.

It suffices to consider complete graphs because we can obtain any subgraph of a complete graph by setting some of the edge weights to $0$.  Let $n\geq m$ be the number of vertices in the complete graph.

Let us consider the determinant of a minor of the Laplacian.  Expanding the determinant by permutations, we can think of each term as a permutation array which acts as a mask revealing only certain entries of the Laplacian.  Consider how off-diagonal entries from the Laplacian can appear in a permutation array.  Either these entries must form a cycle, which cannot happen for a forest (in fact such terms will cancel since re-orienting the cycle gives a sign reversing involution, explaining why only forests appear), or off-diagonal entries correspond to cases where the row and column removed do not match.
The sign that a forest contributes to the determinant includes the sign of the permutation associated to its permutation array and the signs of the entries in the permutation array. The on-diagonal entries of the Laplacian are by definition positive, so negative signs from entries of the array can only be introduced where row and column indices do not match.

\begin{example}
Consider $\det(L_{13,12})$. This gives us forests with two trees, one containing the vertex $1$ and one containing the vertices $2$ and $3$. The tree containing the vertex $1$ has exactly one vertex from the rows removed ($1$) and one vertex from the columns removed (in this case, also $1$). The other tree also contains exactly one vertex from the rows removed ($3$) and one vertex from the columns removed ($2$). 

The sign of each such forest is negative. The forest coming out of the Laplacian uses only entries on the diagonal (again, off-diagonal entries will result in cycles). When removing rows $1$ and $3$ and columns $1$ and $2$, the resulting matrix has the index of every row and column matching except for row $2$ and column $3$, which are now the first row and column in the new matrix. This means that the entry on the diagonal for row $2$ and column $3$ is $-a_{23}$ instead of the positive sum on the diagonal of matching indices. Because the entry is negative and it is the only negative entry at play in such a forest, the sign of each forest of this kind is negative.
\end{example}

Now that we have discussed the signs of forests resulting from the matrix tree theorem, let us look at the signs arising from the determinants in the column expansion identities. Again, we begin by replacing the $M$ in the column expansion identities with a Laplacian with one row and one column already removed.

\begin{definition}
The identity obtained by replacing $M$ in the column expansion identity $j$ with $L_{r,c}$ is the spanning forest identity $\mathcal{L}_{r,c}(j)$.  The right hand side of this identity will be notated $L_{r,c}(j)$.
\end{definition}

\begin{lemma}\label{lem form}
  $\mathcal{L}_{r,c}(j)$ is $(1,1,\ldots, 1)(1,2,\ldots, m) = L_{r,c}(j)$ and $L_{r,c}(j)$ is a sum of $AB$ pairs.
\end{lemma}

\begin{proof}
  Plugging in $L_{r,c}$ for $M$ in the column expansion identity we get
  \begin{align*}
    \det(L_{r,c})\det(L_{[k+1],[k+1]})=&\sum_{i=1}^{r-1}(-1)^{i+j}\det(L_{[k+1]\setminus\{i\},[k+1]\setminus\{j\}})\det(L_{\{i,r\},\{j,c\}})\\&+\sum_{i=r}^{k}(-1)^{i+j}\det(L_{[k+1]\setminus\{i+1\},[k+1]\setminus\{j\}})\det(L_{\{i+1,r\},\{j,c\}})
  \end{align*}
  if $j<c$ and
  \begin{align*}
    \det(L_{r,c})\det(L_{[k+1],[k+1]})=&\sum_{i=1}^{r-1}(-1)^{i+j}\det(L_{[k+1]\setminus\{i\},[k+1]\setminus\{j+1\}})\det(L_{\{i,r\},\{j+1,c\}})\\&+\sum_{i=r}^{k}(-1)^{i+j}\det(L_{[k+1]\setminus\{i+1\},[k+1]\setminus\{j+1\}})\det(L_{\{i+1,r\},\{j+1,c\}}).
  \end{align*}
  if $c\leq j$.
  
  By the matrix tree theorem, $\det(L_{r,c})$ gives the polynomial of all spanning trees of the graph. Writing this in our partition notation $\det(L_{r,c}) = (1,1,\ldots, 1)$.  By the all minors matrix tree theorem, $\det(L_{[k+1],[k+1]})$ gives the spanning forest polynomial where each marked vertex is in a different tree, that is $(1,2,\ldots, m)$.

  Now consider the right hand side.  Terms of the form $\det(L_{[k+1]\setminus\{i\},[k+1]\setminus\{j\}})$ have every column except for column $j$ removed from the Laplacian. Applying the all minors matrix tree theorem, every marked vertex except for vertex $j$ must be in a separate forest, and vertex $j$ must be in the same forest as vertex $i$, hence this is an $A$ partition.  The $B$ partition likewise comes from the terms of the form $\det(L_{\{i,r\},\{j,c\}})$ which by the all minors matrix-tree theorem give spanning forests with two trees.
\end{proof}

\begin{lemma}\label{row}
The row already removed in the Laplacian does not impact the resulting forest identity. That is to say, if we fix a column identity $j$, then replacing $M$ in that column identity with $L_{r,c}$ will give us the same forest identity as replacing $M$ with $L_{r',c}$ for a fixed $c$ and any $r,r'\in [k+1]$.
\end{lemma}

\begin{proof}
  Begin with a fixed $j$ and $c$. We will assume that $j<c$ for this proof, and end with the adjustment to be made if $c\leq j$. Let $r$ be an arbitrary value between $1$ and $k+1$. We will begin by determining the $AB$ pairs possible in $L_{r,c}(j)$.  As in the previous proof we have that $\mathcal{L}_{r,c}(j)$ is
\begin{align*}
    \det(L_{r,c})\det(L_{[k+1],[k+1]})=&\sum_{i=1}^{r-1}(-1)^{i+j}\det(L_{[k+1]\setminus\{i\},[k+1]\setminus\{j\}})\det(L_{\{i,r\},\{j,c\}})\\&+\sum_{i=r}^{k}(-1)^{i+j}\det(L_{[k+1]\setminus\{i+1\},[k+1]\setminus\{j\}})\det(L_{\{i+1,r\},\{j,c\}}).
\end{align*}

The $A$ partition in an $AB$ pair in the forest identity comes from $\det(L_{[k+1]\setminus\{i\},[k+1]\setminus\{j\}})$. Applying the all minors matrix tree theorem, every marked vertex except for vertex $j$ must be in a separate forest, and vertex $j$ must be in the same forest as vertex $i$ (or as $i+1$ if we are in the second sum). Viewing this as an $A$ partition, this means every vertex is in a separate part except for vertex $j$, which must be in a pair with one other vertex.
Let $p_v^A$ be the part that vertex $v$ is in for the $A$ partition and similarly for $p_v^B$.  
The $A$ partition is completely defined by $p_j^A$, since all other vertices must be in a part by themselves, and so we will also write $p^A_j$ for the other vertex in this part.  
The $B$ partition, on the other hand, has only two parts. Without loss of generality, let us call the part that vertex $j$ is in $1$ in our partition notation.  We will write $p_j^B=1$ to indicate this.

To determine what $AB$ pairs are possible within $L_{r,c}(j)$, let us specifically look at the $AB$ pairs where the partition $A$ is defined by $p^A_j=\ell$. There are two options for how this $A$ partition appears from $\det(L_{[k+1]\setminus\{i\},[k+1]\setminus\{j\}})$: one is that $i=j$ (or $i+1=j$ if we are in the second sum), in which case $j$ can pair with any vertex including $\ell$. The second is where $i=\ell$ (or $i+1=\ell$ if we are in the second sum), in which case $j$ must pair with $\ell$. If $i=j$ (or $i+1=j$), then the determinant that gives the $B$ partition is $\det(L_{\{j,r\},\{j,c\}})$. Using the matrix tree theorem, we see that $j$ and $r$ must be in separate partitions. Since we have called $p_j^B=1$, then we must have $p_r^B=2$. Similarly, $p_c^B=2$. These are our only restrictions, so any other vertex (including $\ell$) can be in either part.

If instead we have that $i=\ell$ (or $i+1=\ell$ if we are in the second sum), then the determinant that gives the $B$ partition is $\det(L_{\{\ell,r\},\{j,c\}})$. Again, $j$ and $c$ must be in different parts, so again, $p_c^B=2$. Then we have two possibilities for the rows: either $p_\ell^B=1$ and $p_r^B=2$, or $p_\ell^B=2$ and $p_r^B=1$. These are our only restrictions, so any other vertex can be in either part.

In summary, starting with the assumption that $p^A_j=\ell$ and $p_j^B=1$, then $p_c^B$ must always equal $2$, and all other vertices (aside from $\ell$ and $r$) can always be in either part in $B$. So our possibilities are that $p_r^B=2$ and $p_\ell^B=1$; $p_r^B=2$ and $p_\ell^B=2$; and that $p_r^B=1$ and $p_\ell^B=2$. Interestingly, the case in which $p_r^B=2$ and $p_\ell^B=1$ shows up twice: once when $i=j$ and once when $i=\ell$.

Let us look at this case more closely. We claim that the signs in the two cases when $p_r^B=2$ and $p_\ell^B=1$ are opposite. Then these instances would cancel out, and this partition would actually not appear in the final forest identity. 

Let us start in the first sum, that is assuming that $j,\ell<r$. Then when $i=j$, we have signs coming from three places: $(-1)^{i+j}$, and each of the two determinants. Because $i=j$, the $(-1)^{i+j}$ will just contribute a positive sign. In the first determinant, $\det(L_{[k+1]\setminus\{j\},[k+1]\setminus\{j\}})$, the indexing on both the rows and columns match, so the sign is positive. In the second determinant, $\det(L_{\{j,r\},\{j,c\}})$, we do not necessarily have that $r$ and $c$ match. However, we can switch rows until the existing row $c$ is in the same place as the existing column $r$. Then the base determinant would be positive, and the sign would be introduced by the number of times we swap rows to line up the row $c$ with the column $r$, and by the signs of the entries on the diagonal. Since in our assumptions, $i=j$ are both less than $r$ and $c$, all rows in between $r$ and $c$ are still in the matrix, so we need to switch $|c-r-1|$ times to get row $c$ in the same position as column $r$. This means our sign from switching rows is $(-1)^{c-r-1}$. Once we have done the swaps, every entry is on the diagonal, but the entry in row $c$, column $r$ is negative since it did not originally come from the diagonal. Thus our sign for the second determinant is $(-1)^{c-r}$, so our overall sign for the partition when $i=j$ is $(-1)^{i+j+c-r}=(-1)^{c-r}$.

In contrast, when $i=\ell$, we still have $(-1)^{i+j}$, but our first determinant no longer has matching indices of removed rows and columns. That is to say, although row $i$ and column $j$ do line up (since all other rows and columns before $k+2$ have been removed), that entry did not originally come from a diagonal, so the first determinant contributes a negative sign. In the second determinant, since we are specifically looking at the case where $p_r^B=2$ and $p_\ell^B=1$, we want to pair the existing column $r$ with row $c$, and the existing column $i=\ell$ with row $j$. As discussed in the previous paragraph, swapping rows so the $r$ and $c$ match up and taking into account the negative entry gives a sign of $(-1)^{c-r}$. Here we must also swap row $j$ so that it lines up with column $i$. Since these rows do not interact with $r$ or $c$, we will similarly get a sign of $(-1)^{i-j}$. Taken all together, this forest will have a sign of $(-1)^{i+j+1+c-r+i-j}=(-1)^{c-r+1}$. Notice that this is the opposite of the sign when the forest comes from $i=j$, so these two forests cancel out.

We did this assuming that $\ell, j<r$. We could also have that $\ell<r<j$. In this case, we are dealing with the second sum when $i+1=j$. The effect on the sign is that we still have $(-1)^{i+j}$, but now $i$ and $j$ are opposite parity instead of the same parity. This will contribute a minus sign. In the first determinant, the indexing still matches, so we still get a positive sign. In the second determinant, since row $i+1$ is removed and is in between $r$ and $c$, then row $c$ has to switch with one fewer row to get to position $r$. This means that instead of the overall sign of the second determinant being $(-1)^{c-r}$, it will be $(-1)^{c-r-1}$. However, taken with the negative sign contributed by $(-1)^{i+j}$, we still have a sign of $(-1)^{c-r}$ overall for the forest pair. In the $i=\ell$ case, we are still in the first sum since $\ell<r$, and so the signs work out the same as the first time we did it, giving a sign of $(-1)^{c-r+1}$. Again, the signs are opposite, and the forest pairs cancel out.

If instead we have $j<r<\ell$, then when $i=j$, we are in the first sum and the signs work out the same as the first time we did it, so the sign on that forest pair is $(-1)^{c-r}$. When $i+1=\ell$ we are in the second sum. Then we still have $(-1)^{i+j}$ contributed by the beginning and a minus sign contributed by the first determinant. In the second determinant, there are two changes: first, switching row $j$ to the $\ell$ position will require one less swap since it must pass by the empty $r$ row. However, second, it will require one extra swap since it is trying to get to position $i+1$ instead of position $i$. Taken together, this gives us the same number of swaps, and therefore the same sign as before, namely $(-1)^{c-r+1}$. Again, our signs are opposite and the two cancel out.

Finally, we could have $r<j,\ell$. When $i+1=j$, we are in the second sum, which we have already shown to give a sign of $(-1)^{c-r}$. When $i+1=\ell$, we are also in the second sum. We do still need an extra swap to get to position $i+1$ instead of position $i$. We also do still need one less swap. This time it is not for $j$ to pass by the empty $r$ row, since $j$ is bigger than $r$, but rather for $c$ to pass by the empty $j$ row since $j$ is smaller than $c$. Regardless, the sign still comes out to $(-1)^{c-r+1}$, and again, the forest pairs are of opposite signs and cancel.

We have now proven our claim that when $p_r^B=2$ and $p_\ell^B=1$, these partitions end up showing up twice, each of opposite sign, and cancelling each other out. That means that in $L_{r,c}(j)$, when $p^A_j=\ell$ and $p_j^B=1$, we only have two possibilities: $p_r^B=1$, $p_c^B,p_\ell^B=2$, and everything else could be either; or $p_r^B=2$, $p_c^B,p_\ell^B=2$, and everything else could be either. Because $p_r^B$ can either be $1$ or $2$, we actually only have one scenario: If we assume that $p_j^A=\ell$ and $p_j^B=1$, then we must have that $p_c^B,p_\ell^B=2$ and everything else could be either. Notice then, that the options available have nothing to do with the row selected, they are only dictated by the column $c$ originally removed and the column identity $j$ that is used. As a result, assuming that $j<c$, we have shown that the row removed does not impact which monomials appear in the identity.

If instead we have that $c\leq j$, the new identity becomes 
\begin{align*}
    \det(L_{r,c})\det(L_{[k+1],[k+1]})=&\sum_{i=1}^{r-1}(-1)^{i+j}\det(L_{[k+1]\setminus\{i\},[k+1]\setminus\{j+1\}})\det(L_{\{i,r\},\{j+1,c\}})\\&+\sum_{i=r}^{k}(-1)^{i+j}\det(L_{[k+1]\setminus\{i+1\},[k+1]\setminus\{j+1\}})\det(L_{\{i+1,r\},\{j+1,c\}}).
\end{align*}

The only thing changed here is that we are largely just indexing by $j+1$ while the sign $(-1)^{i+j}$ at the beginning of each sum does not change to $j+1$. This merely reverses our signs in our argument showing that the two instances of $p_r^B=2$ and $p_\ell^B=1$ cancel each other out. Since the specific sign there did not matter, just that the signs were opposite, this unilateral sign change does not impact the result.
\end{proof}

Because we have just shown that the row originally removed from the Laplacian does not matter, for simplicity of indexing we will usually either match the index of the row removed with the index of the column removed, that is, that $r=c$, or take $r=m$.

\begin{definition}
Fix $c,j\in [k+1]$. We will call an $AB$ pair a \textbf{permissible monomial} if:

\begin{itemize}
    \item In the partition $A$, $j$ is in a part with another element, let us call it $\ell$. All other elements aside from $j$ and $\ell$ are in parts by themselves.
    \item In the partition $B$, $j$ is in one part, and $c$ and $\ell$ are in the other part. All other elements may be in either part.
\end{itemize}
\end{definition}

\begin{lemma}\label{exactmonomial}
Fix $c,j\in [k+1]$. Then the right hand side of $L_{c,c}(j)$ consists of exactly one copy of every permissible monomial.
\end{lemma}

\begin{proof}
We showed in the proof of Lemma \ref{row} that only permissible monomials appear in $L_{c,c}(j)$. Then we need to show that each permissible monomial must appear at least once, and does not appear more than once.

Suppose that we have a permissible monomial $AB$ such that $j$ is paired with $\ell$ in $A$. We showed in the proof of Lemma \ref{row} that this monomial would appear once when $i=j$: because our row and column removed from the original Laplacian are both $c$ and $i=j$, the indices of all removed rows and columns match. This means that every monomial will appear one time when $i=j$, and will have a positive sign. Then we just need to show that the permissible monomial does not appear more than once.

As shown in the proof of Lemma \ref{row}, the only other way that $j$ can be paired with $\ell$ in $A$ is if $i=\ell$ (or $i+1=\ell$ if $\ell>c$). However, when this occurs, the determinant that yields the $B$ partition is $\det(L_{\{\ell,c\},\{j,c\}})$. Because $c$ is the index of both the row and column originally removed, this would require $\ell$ and $j$ to be in the same partition, resulting in a monomial that is not permissible. Thus we have shown that the only way to get a permissible monomial is when $i=j$, so permissible monomials cannot appear more than once.
\end{proof}

Now that we have a good understanding of the monomials associated with the column expansion identities, we can look at the overall interpretation of the column expansion identities. We will give two definitions to make our interpretation easier to verbalize.

\begin{definition}
For a fixed integer $k$, a \textbf{$k$-forest} is a forest with $k$ trees.
\end{definition}

\begin{definition}
A pair of forests is called \textbf{forbidden} if two or more marked vertices are in the same tree in both pairs. In the language of partitions, a pair of partitions is \textbf{forbidden} if two or more marked vertices are in the same part in both partitions.
\end{definition}

In the case of $AB$ pairs, note that permissible pairs are all non-forbidden, but that being permissible is stronger than being non-forbidden due to the additional constraint on the location of $c$ among the $B$ partition.

\begin{example}
The partition pair $A_5B_2$ is forbidden because $A_5=(1,2,3,2)$ has vertices $2$ and $4$ in the same part, and $B_2=(1,1,2,1)$ also has vertices $2$ and $4$ in the same part.
\end{example}

\begin{proposition}
Fix $c$ and $j$ such that $c,j\in [k+1]$. The left hand side of the column expansion identities corresponds to graph pairs, one of which is a tree and one of which is a $(k+1)$-forest. The right hand side of the column expansion identities corresponds to pairs of non-forbidden forests, one of which is a $2$-forest and one of which is a $k$-forest, such that $c$ and $j$ are in different trees in the $2$-forest.
\end{proposition}

\begin{proof}
Again, let us look at the column expansion identities when we replace $M$ by $L_{c,c}$. By a direct application of the matrix tree theorem, the left hand side of the column expansion identities yields graph pairs, one of which is a tree and one of which is a $(k+1)$-forest. The right hand side yields graph pairs, one of which is a 2-forest and one of which is a $k$-forest. 

Because the left hand side does not involve any minus signs, the interpretation of a tree and a $(k+1)$-forest suffices. On the right hand side, some of the pairs of forests are subtracted off. Lemma \ref{exactmonomial} tells us that exactly one copy of each permissible monomial appears in the right hand side. The $A$ partition corresponds to the $k$-forest, and the $B$ partition corresponds to the $2$-forest. By definition of a permissible monomial, if $j$ and $\ell$ are in the same part in $A$, they must be in separate parts in $B$, so all corresponding forests are non-forbidden. Additionally, since $j$ and $c$ are always in different parts in the $B$ partition in permissible monomials, they must be in different trees in the $2$-forest. Thus a permissible monomial corresponds to non-forbidden forests in which $c$ and $j$ are in different trees in the $2$-forest, and so Lemma \ref{exactmonomial} proves the right hand side of our result.

\end{proof}

\subsection{Combinatorial Proof of Column Expansion Identities}\label{subsec combi}

The proof given in section 2.2 of the column expansion identities is an algebraic proof based on the Dodgson/Muir identity. However, we can view this as a combinatorial proof as well by using the combinatorial proof of the Dodgson/Muir identity given by the first author \cite{dennis}. This proof consists of an algorithm called the generalized Red Hot Potato algorithm that matches a set of $k$ ordered forests (one tree rooted at zero and $k-1$ forests each rooted at vertices $0$ through $k$) to a set of $k$ ordered forests, each a $k$-forest rooted at $0,1,\dots,i-1,i+1,\dots, k$. Figure \ref{genrhpfig} gives a schematic. The Red Hot Potato algorithm accomplishes this by swapping edges back and forth amongst the set of $k$ graphs. 

\begin{figure}
    \centering
    \includegraphics[scale=.4]{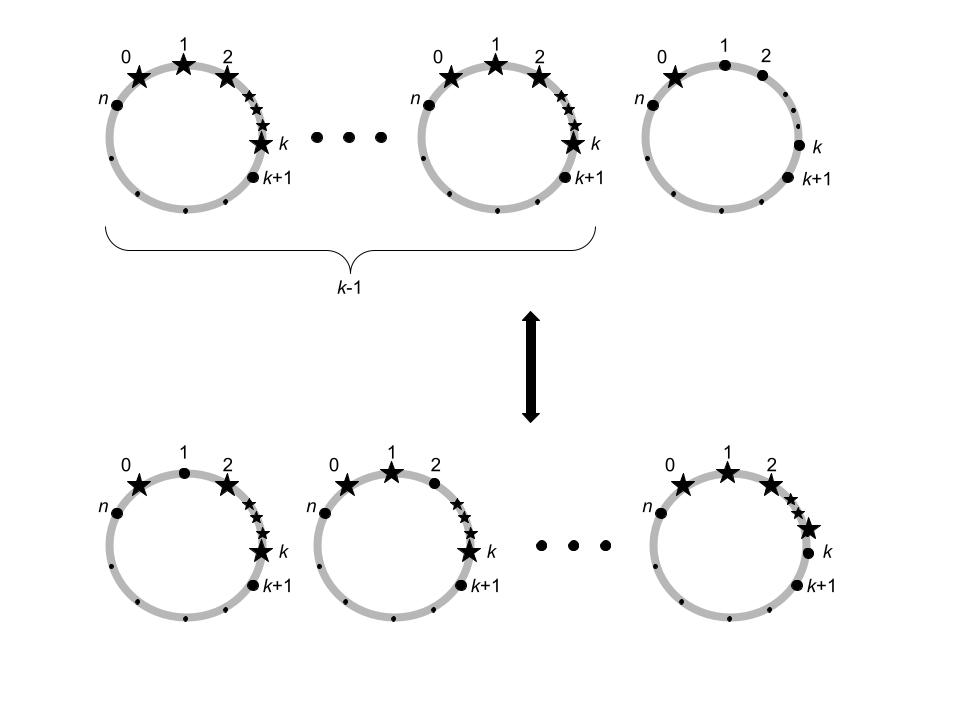}
    \caption{The generalized Red Hot Potato algorithm is a bijection between the two sets illustrated above. The stars represent roots of trees and have no edges coming out of them. The simple vertices each have one edge leaving.}
    \label{genrhpfig}
\end{figure}

For the column expansion identity, we start with a pair of graphs from the left-hand side, one of which is a tree and one of which is a $(k+1)$-forest. When applying the column expansion identity to the problem of finding quadratic spanning forest identities, we will be thinking of the resulting forests as undirected, but the Red Hot Potato algorithm requires directed forests. However, since the column expansion identity itself is coming from a set of matrix determinants, for the purposes of proving the column expansion identity, we can think of these graphs as directed by replacing $M$ with a Laplacian that already has the $0$th row and $0$th column removed. In this way, we will start with a tree rooted at $0$ and a $(k+1)$- forest rooted at $0$ through $k$. To apply the generalized Red Hot Potato algorithm, we need $k-1$ forests rooted at $0$ through $k$. We will union in $k-2$ more forests, all of which consist of no edges coming out of vertices $0$ through $k$ and one edge from vertex $\ell$ to vertex $0$ for all $\ell>k$ (Figure \ref{colexpsetup}).

\begin{figure}
    \centering
    \includegraphics[scale=.4]{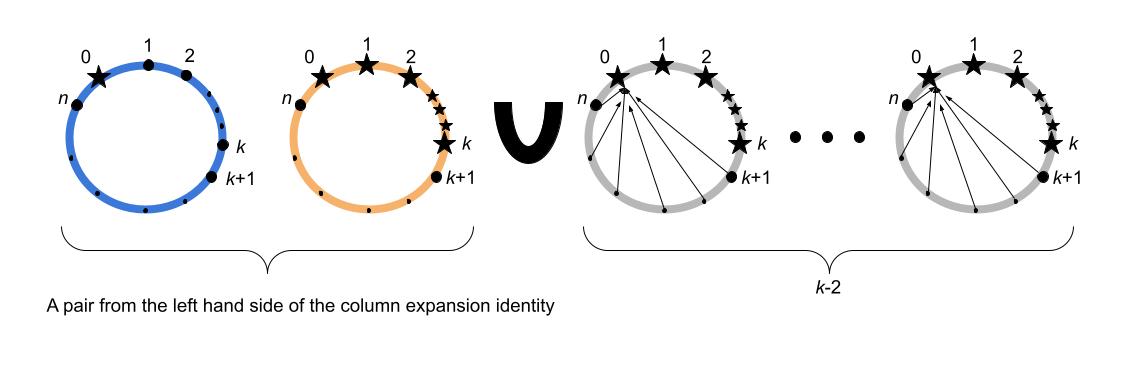}
    \caption{The pair on the left come from the left-hand side of the column expansion identity. The non-starred vertices in these two graphs each has one edge coming out of it that could go to any other vertex. The $k-2$ forests on the right are ``dummy" forests to allow the application of the Red Hot Potato algorithm.}
    \label{colexpsetup}
\end{figure}

The forests need to be ordered for the algorithm to work. We will order these so that our original $(k+1)$-forest is the $j$th one out of the all of the $(k+1)$-forests, where $j$ is the fixed column that we are expanding along in the column expansion identity.

We now have a tree and $k-1$ ordered $(k+1)$-forests, which is what we need to perform the generalized Red Hot Potato algorithm. Do so. We know that we will finish with $k$ ordered $k$-forests. In particular, it turns out that, with the exception of the $j$th forest and the $k$th forest (which were what we originally started with), forest $i$ will have an edge out of $i$, which originated from the original tree, and edges out of vertices $k+1,\dots n$ all going to $0$. In fact,  the edges in the ``dummy" forests do not actually move during the algorithm (Figure \ref{colexpreorder}).

\begin{figure}
    \centering
    \includegraphics[scale=.3]{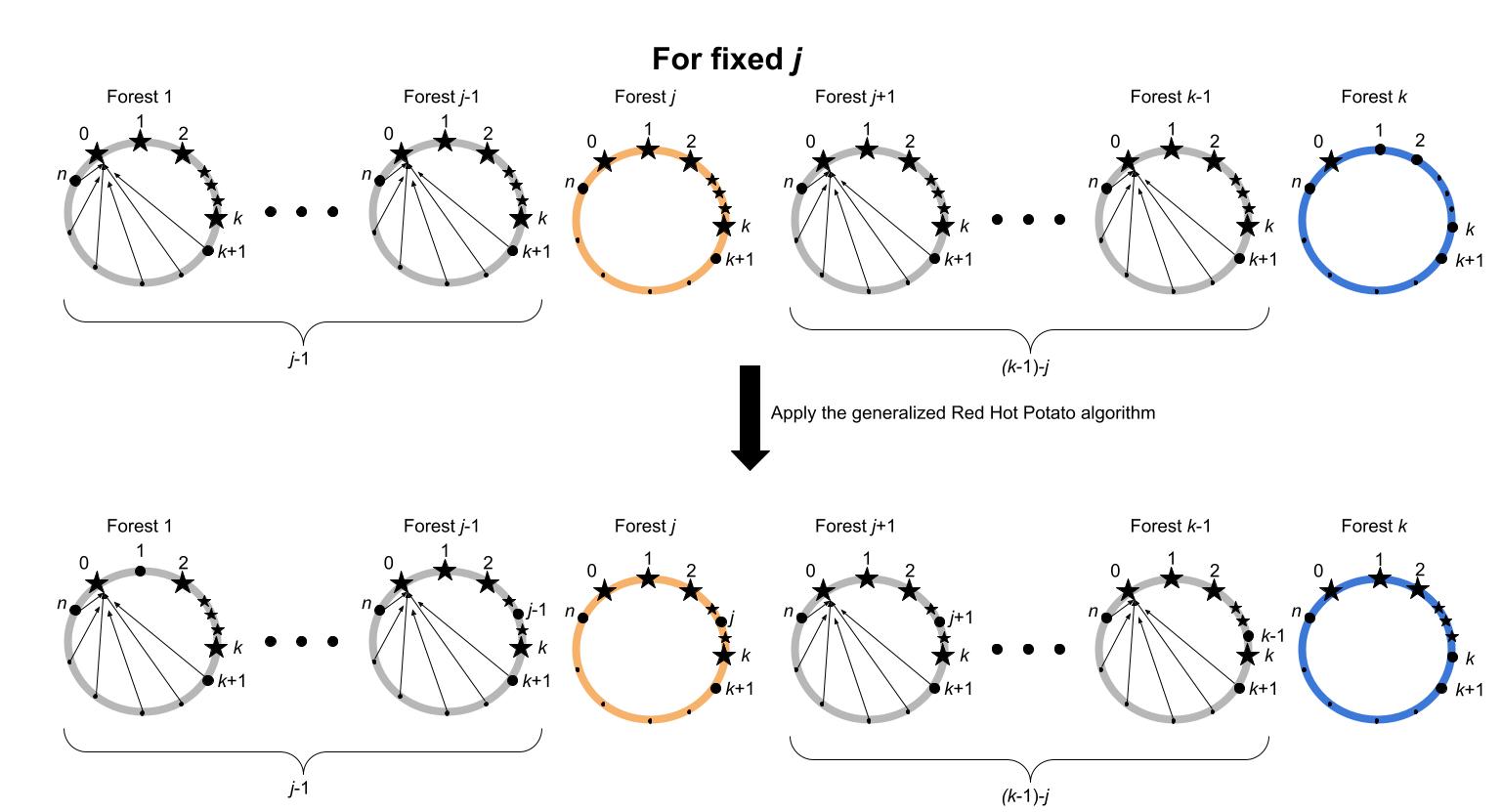}
    \caption{The top is the result of re-ordering our $k$ forests. The bottom is the set of $k$ forests after applying the generalized Red Hot Potato algorithm. Any non-starred vertex without an edge specifically drawn in has one edge coming out of it that could go to any other vertex.}
    \label{colexpreorder}
\end{figure}

    Finally, we are going to remove the $j$th forest. This corresponds with $\det(M_{[k]\setminus\{j\},[k]\setminus\{j\}})$. That is the second determinant in the column expansion identity. It must have $j$ as both the row and column in the determinant because when the indices do not match, the forest is forbidden and gets subtracted. Once we have removed the $j$th forest, we can think of $j$ as a special root similar to $0$: there are no edges coming out of it in any of the forests (since originally there was only one edge total coming out of $j$, and it ended up in the $j$th forest), so we are effectively ignoring it. Then we have $k-1$ ordered $(k-1)$-forests (if we ignore $j$), each forest $i$ with no edge out of $0,\dots k$ except for an edge out of $i$. This is what we need to do the generalized Red Hot Potato algorithm, so we do it. 
    
    We end with one tree with $0$ as a root (technically this is actually a $2$-forest with $0$ and $j$ as roots), and $k-2$ forests with no edges coming out of $0,1,\dots k$. In particular we claim that these forests have all the edges pointed to $0$. The ``tree" is $\det(M_{j,j})$ in the right-hand side of the identity. The remaining $k-2$ forests are identical to the $k-2$ forests that we originally added in, so we remove them again, leaving us with our ``tree" that corresponds to $\det(M_{j,j})$ and our $j$th forest that corresponds to $\det(M_{[k]\setminus\{j\},[k]\setminus\{j\}})$. This is the right hand side of the column expansion identity (Figure \ref{colexpfinal}). Since all we have actually done is apply the generalized Red Hot Potato algorithm twice, and we already know that this is a bijection, then our whole process was a bijection and we have proved the column expansion identity combinatorially. It is not hard to prove that the $k-2$ forests that we added at the beginning end up the same at the end (i.e. that the edges going from $\ell$ to $0$ for all $\ell>k$ do not get moved around). However, this involves going in depth into the definition of the generalized Red Hot Potato algorithm, which is outside the scope of this paper.
    
    \begin{figure}
        \centering
        \includegraphics[scale=.3]{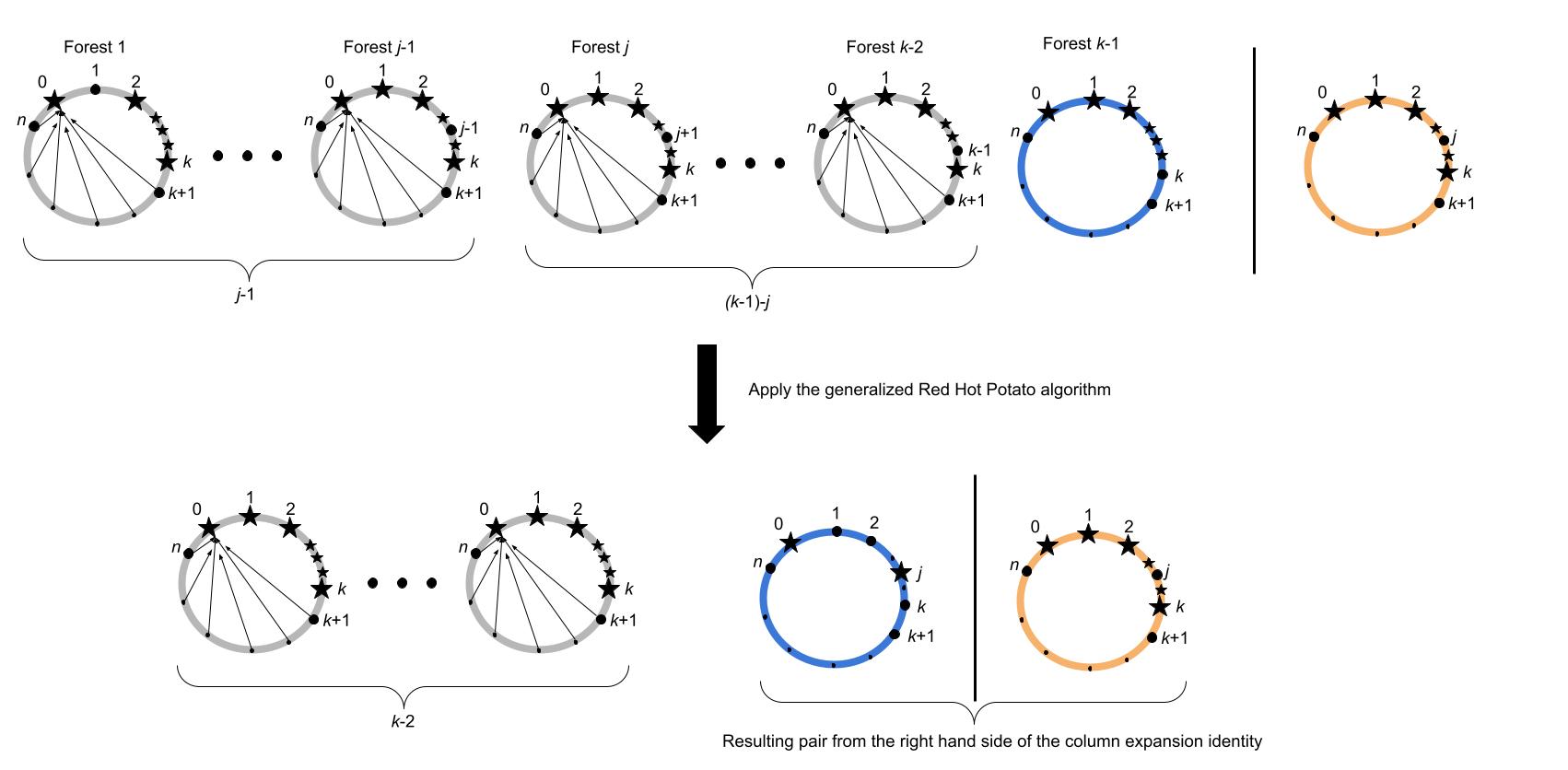}
        \caption{We remove the orange forest $j$ (placed here on the right hand side of the bar) and apply the algorithm to the remaining forests. The resulting blue and orange forests are a pair from the right hand side of the column expansion identity, and the extra $k-2$ gray forests are the same ``dummy" forests that we added at the beginning.}
        \label{colexpfinal}
    \end{figure}

\section{Quadratic Spanning Forest Identities}

Let us begin our discussion of quadratic spanning forest identities by looking at the case arising from the classical Dodgson identity. This is the case when the number of marked vertices is $m=3$.
We will show that the Dodgson identity viewed in this way is consistent with Conjecture \ref{VYConj}.  In some ways this case is unusual because certain things which are distinct in general are not distinct in this situation.

When we look at the column expansion identities for $k=2$, we get the following two identities: 
\begin{align*}
\det M \det M_{\{1,2\},\{1,2\}} &= \det M_{\{1\}, \{1\}} \det M_{\{2\}, \{2\}} - \det M_{\{2\}, \{1\}} \det M_{\{1\}, \{2\}}\\
\det M \det M_{\{1,2\},\{1,2\}} &= \det M_{\{2\}, \{2\}} \det M_{\{1\}, \{1\}} - \det M_{\{1\}, \{2\}} \det M_{\{2\}, \{1\}}
\end{align*}

Notice that both of these identities are the Dodgson identity, with just the order of the determinants switched. Typically, the order would not matter since multiplication is commutative. However, for the purposes of the conjecture we are not counting determinantal identities per se, but rather identities formed by sums of $AB$ pairs. As we will see below, when we translate these into identities of $A$ and $B$ partitions, we get different $AB$ pairs from different orders.  As described in Section~\ref{sec comb int}, we can replace $M$ with a Laplacian with a row and column already removed. Since by Lemma \ref{row} the row removed does not affect the outcome, when $m=3$ we merely need to decide which of the first three columns of the Laplacian to remove for each of the two identities. We go into more detail about how to obtain identities from specific Laplacians in the next section; following that method, we have six identities:

\begin{align*}
   \mathcal{L}_{3,1}(1): (1,1,1)(1,2,3)&=(1,2,2)(1,1,2)+(1,2,1)(1,1,2)+(1,2,1)(1,2,2)\\
   \mathcal{L}_{3,1}(2):(1,1,1)(1,2,3)&=(1,1,2)(1,2,2)+(1,1,2)(1,2,1)+(1,2,2)(1,2,1)\\
   \mathcal{L}_{3,2}(1): (1,1,1)(1,2,3)&=(1,2,2)(1,1,2)+(1,2,2)(1,2,1)+(1,2,1)(1,1,2)\\
   \mathcal{L}_{3,2}(2):(1,1,1)(1,2,3)&=(1,1,2)(1,2,2)+(1,2,1)(1,2,2)+(1,1,2)(1,2,1)\\
   \mathcal{L}_{3,3}(1):(1,1,1)(1,2,3)&=(1,1,2)(1,2,1)+(1,2,2)(1,2,1)+(1,2,2)(1,1,2)\\
   \mathcal{L}_{3,3}(2): (1,1,1)(1,2,3)&=(1,1,2)(1,2,2)+(1,2,1)(1,1,2)+(1,2,1)(1,2,2)
\end{align*}

As identities of spanning forest polynomials all six of these are the same identity. However, it is important to note that Conjecture \ref{VYConj} refers to identities that are written in the form of $AB$ partition pairs where $A$ partitions consist of $m-1$ parts and $B$ partitions consist of $2$ parts. In this case where $m=3$, $A$ partitions are the same as $B$ partitions. Thus, in this case, the order of the partitions matters. That is to say $(1,2,1)(1,1,2)$ is a different $AB$ partition pair than $(1,1,2)(1,2,1)$, since in the first case the $A$ partition is $(1,2,1)$ and the $B$ partition is $(1,1,2)$ while in the second case the reverse is true.  All six of the identities given above are distinct by this measure.

Explicitly, if we write $A_1 = (1,2,2) = B_1$, $A_2 = (1,2,1) = B_2$, and $A_3 = (1,1,2) = B_3$ then the six identities are:

\begin{align*}
   \mathcal{L}_{3,1}(1): (1,1,1)(1,2,3)&=A_1B_3+A_2B_3+A_2B_1\\
   \mathcal{L}_{3,1}(2):(1,1,1)(1,2,3)&=A_3B_1+A_3B_2+A_1B_2\\
   \mathcal{L}_{3,2}(1): (1,1,1)(1,2,3)&=A_1B_3+A_1B_2+A_2B_3\\
   \mathcal{L}_{3,2}(2):(1,1,1)(1,2,3)&=A_3B_1+A_2B_1+A_3B_2\\
   \mathcal{L}_{3,3}(1):(1,1,1)(1,2,3)&=A_3B_2+A_1B_2+A_1B_3\\
   \mathcal{L}_{3,3}(2): (1,1,1)(1,2,3)&=A_3B_1+A_2B_3+A_2B_1
\end{align*}

If there were an identity of the form $(1,1,1)(1,2,3) = \sum \alpha_{i,j}A_iB_j$ for some coefficients $\alpha_{i,j}$ which was not in the span of the identities above, then this new identity would be true on every graph.  In particular it would be true on the complete graph on three vertices.  Labelling the edge from $2$ to $3$ by $a$, from $1$ to $3$ by $b$ and from $1$ to $2$ by $c$, as in Figure \ref{fig triangle}, we would have $(1,2,2) = bc$, $(1,1,2) = ab$, $(1,2,1)=ac$, $(1,1,1) = a+b+c$, and $(1,2,3) = abc$.  So $(1,1,1)(1,2,3) = (a+b+c)(abc) = a^2bc + ab^2c + abc^2$.  However, each of the terms in this expansion can only be factored into squarefree monomials of degree 2 (which the required partitions give as their polynomials) in one way: $(a+b+c)(abc) = (ab)(ac) + (ab)(bc) + (ac)(bc)$, and assigning these factors as $A$s and $B$s we get exactly the six identities above.

However, viewing the identities as polynomials in the variables $A_i$ and $B_j$, the six identities given above are not linearly independent: the right hand sides of each pair that comes from the same Laplacian have the same sum (i.e. $L_{3,1}(1)+L_{3,1}(2)=L_{3,2}(1)+L_{3,2}(2)=L_{3,3}(1)+L_{3,3}(2)$, and all three of these sums equals $A_1B_2+A_1B_3+A_2B_1+A_2B_3+A_3B_1+A_3B_2$). 

To check the conjecture in this case it remains to count the degrees of freedom.  First homogenize so as to translate the solutions to the origin where they form a subspace -- we can do so by subtracting any of the six equations from the others, leaving five equations.  Then, $L_{3,1}(1)+L_{3,1}(2)=L_{3,2}(1)+L_{3,2}(2)=L_{3,3}(1)+L_{3,3}(2)$ gives two identities, leaving a space of dimension $3=m(m-2)$ as the conjecture states.

Note that when $m>3$ the identities will be different as spanning forest identities not just as sums of $AB$ pairs, since the $A$s will be distinct from the $B$s.  

\medskip

With this example under our belts, it is a good time to rephrase the conjecture more formally.  Let $m$ be an integer at least 3.  Let $a_m$ be the number of set partitions of $\{1,2,\ldots, m\}$ into $m-1$ parts and let $b_m$ be the number of set partitions of $\{1,2,\ldots, m\}$ into 2 parts, and $\{A_i\}_{i=1}^{a_m}$ and $\{B_j\}_{j=1}^{b_m}$ be the sets of these partitions in some order which we now take to be fixed.  The original conjecture asked about the number of free variables in the most general expression of the form $(1,1,\ldots, 1)(1,2,\ldots, m) = \sum_{i,j}\alpha_{i,j}A_iB_j$ which is true on any graph when the set partitions are interpreted as spanning forest polynomials.

This is asking for the solution to an inhomogeneous linear system, so homogenizing by subtracting any particular solution (and we have many explicit particular solutions since each column expansion identity gives one by Lemma~\ref{lem form}), the question is asking about the dimension of the vector space of expressions of the form $0=\sum_{i,j}\alpha_{i,j}A_iB_j$ which are true on any graph when the set partitions are interpreted as spanning forest polynomials.  More formally we can rephrase this as follows.

Let $V_m$ be the vector space generated by monomials $A_iB_j$.  Define the subspace $X_m$ of $V_m$ as follows. For any graph $G$ with $m$ marked vertices we have a linear map from $V$ to a vector space of polynomials given by evaluating each set partition as its corresponding spanning forest polynomial on $G$.  The kernel of this map is a subspace of $V_m$ and the intersection of these kernels running over all graphs with $m$ marked vertices also gives a subspace; this latter subspace is $X_m$.

\begin{conjecture}[Conjecture~\ref{VYConj} rephrased]\label{conj rephrased}
The dimension of $X_m$ is $m(m-2)$.  

Furthermore, there is at least one identity of the form $(1,1,\ldots, 1)(1,2,\ldots, m) = \sum_{i,j}\alpha_{i,j}A_iB_j$ which is true on any graph with $m$ marked vertices, and hence the number of free variables in the most general such expression is the dimension of $X_m$.
\end{conjecture}

As well as proving that the dimension matches the conjecture we will give an explicit basis built from column expansion identities for each $m$.

\subsection{Quadratic spanning forest identities with $m=4$ marked vertices}\label{subsec 4 vert}

The case with $m=4$ marked vertices is more representative of the general case and is also the case studied in \cite{vlasev}.  Recall the indexing for the $A_i$ and $B_j$ for $m=4$ as given in Section~\ref{subsec set up} 

$$\begin{array}{ccc}
    A_1=(1,1,2,3) & A_2=(1,2,1,3) & A_3=(1,2,2,3) \\
    A_4=(1,2,3,1) & A_5=(1,2,3,2) & A_6=(1,2,3,3)
\end{array}
$$

$$\begin{array}{ccc}
    B_1=(1,1,1,2) & B_2=(1,1,2,1) & B_3=(1,2,1,1) \\
    B_4=(1,2,2,2) & B_5=(1,1,2,2) & B_6=(1,2,1,2) \\
    & B_7=(1,2,2,1). &
\end{array}
$$

Vlasev and the second author \cite{vlasev} discovered the following identity:

\begin{theorem}\label{VY}
\begin{equation*}
\begin{split}
    (1,1,1,1)(1,2,3,4)=&(1-x_1-x_2)A_4B_1+x_7A_2B_4+(1-x_3-x_2)A_5B_1\\
    &+(1-x_1-x_4)A_6B_1+x_2A_2B_2+(x_3+x_2-x_5)A_3B_2\\
    &+(1-x_1-x_6)A_6B_2+x_1A_1B_3+(x_1-x_7+x_4)A_3B_3\\
    &+(x_1-x_8+x_6)A_5B_3+x_5A_1B_4+(x_1-x_5+x_4)A_3B_5\\
    &+(x_1-x_5+x_6)A_5B_5+x_3A_1B_6+(x_3+x_2-x_7)A_3B_6\\
    &+(1-x_1-x_2+x_8-x_6)A_4B_6+(x_2+x_7-x_4)A_2B_7\\
    &+(1-x_1-x_7+x_8-x_6)A_6B_6+(x_1+x_5-x_3)A_1B_7\\
    &+(1+x_5-x_3-x_2-x_8)A_5B_7\\
    &+(1-x_1+x_7-x_4-x_8)A_6B_7\\
    &+x_8A_4B_4+x_4A_2B_5+x_6A_4B_5
    \end{split}
\end{equation*}

holds for all $x_1,x_2,\dots,x_8$, and all identities of the form $(1,1,1,1)(1,2,3,4) = \sum \alpha_{i,j} A_i, B_j$ are special cases of this one.
\end{theorem}

Though this theorem covers all possible identities of this form, Vlasev and the second author did not have a combinatorial proof, nor a proof that generalized to $m>4$. We will use the column expansion identities to remedy both of these problems.  In this subsection we will consider how to use the column expansion identities to give a more conceptual and in principle combinatorial (thanks to Section~\ref{subsec combi}) reformulation of this identity, while the subsequent subsections will prove the generalization.

In order to obtain an identity, we can use the column expansion identities for $m=4$, replacing $M$ with the Laplacian for a complete graph with a row and column already removed. As established in Lemma \ref{row}, the resulting identity is not impacted by the row removed, so for the purposes of consistency, we will always remove the fourth row.

\begin{example}
We will show in detail how to attain $\mathcal{L}_{4,4}(1)$. We begin by applying the first column expansion identity to the Laplacian with the fourth row and fourth column removed. This gives us:

\begin{align*}
    \det(L_{4,4})\det(L_{1234,1234})&=\det(L_{14,14})\det(L_{234,234})-\det(L_{24,14})\det(L_{134,234})+\det(L_{34,14})\det(L_{124,234})
\end{align*}

We now apply the matrix tree theorem to these determinants. Recall that the generalized version of the matrix tree theorem \cite{chaiken} says that the determinant of the Laplacian with $k$ rows (or columns) removed is given by those forests with $k$ trees so that each tree contains exactly one index of the rows removed, and exactly one index of the columns removed. Thus, for example, $\det(L_{24,14})$ corresponds to forests with two trees, one of which contains the vertex $4$, and one of which contains the vertices $1$ and $2$. Vertex $3$ is not removed from either rows or columns, so it can belong to either tree. In partition notation, $\det(L_{24,14})$ corresponds to $(1,1,-,2)$. If we apply this to all of the determinants in our identity above, we get the following (we switch the right-most plus to a minus because the determinant product is negative):

\begin{align*}
    \mathcal{L}_{4,4}(1): (1,1,1,1)(1,2,3,4)&=(1,-,-,2)(-,1,2,3)-(1,1,-,2)(1,1,2,3)-(1,-,1,2)(1,2,1,3)\\
    &= (B_1+B_4+B_5+B_6)(A_1+A_2+A_4)-(B_1+B_5)A_1-(B_1+B_6)A_2\\
    &=A_4(B_1+B_4+B_5+B_6)+A_1(B_4+B_6)+A_2(B_4+B_5)
\end{align*}
\end{example}

If we follow the same process for all other combinations of column removed and column identity applied, we obtain the following identities (the $x_i$ values under each identity show what each $x_i$ in Theorem \ref{VY} would need to be set to in order to obtain the identity):

\allowdisplaybreaks
\begin{align*}
    \mathcal{L}_{4,4}(1):(1,1,1,1)(1,2,3,4)=A_4(B_1&+B_4+B_5+B_6)+A_1(B_4+B_6)+A_2(B_4+B_5)\\
    &x_1=x_2=0, \text{ other } x_i=1\\
    \mathcal{L}_{4,4}(2):(1,1,1,1)(1,2,3,4)=A_5(B_1&+B_3+B_5+B_7)+A_1(B_3+B_7)+A_3(B_3+B_5)\\
    &x_1=1, \text{ other } x_i = 0\\
    \mathcal{L}_{4,4}(3):(1,1,1,1)(1,2,3,4)=A_6(B_1&+B_2+B_6+B_7)+A_2(B_2+B_7)+A_3(B_2+B_6)\\
    &x_2=1, \text{ other } x_i = 0\\
    \mathcal{L}_{4,3}(1):(1,1,1,1)(1,2,3,4)=A_2(B_2&+B_4+B_5+B_7)+A_1(B_4+B_7)+A_4(B_4+B_5)\\
    &x_1=x_3=0, \text{ other } x_i = 1\\
    \mathcal{L}_{4,3}(2):(1,1,1,1)(1,2,3,4)=A_3(B_2&+B_3+B_5+B_6)+A_1(B_3+B_6)+A_5(B_3+B_5)\\
    &x_1=x_3=1, \text{ other } x_i = 0\\
    \mathcal{L}_{4,3}(3):(1,1,1,1)(1,2,3,4)=A_6(B_1&+B_2+B_6+B_7)+A_4(B_1+B_6)+A_5(B_1+B_7)\\
    &x_i=0\hspace{2mm}\forall i\\
    \mathcal{L}_{4,2}(1):(1,1,1,1)(1,2,3,4)=A_1(B_3&+B_4+B_6+B_7)+A_2(B_4+B_7)+A_4(B_4+B_6)\\
    &x_2=x_4=x_6=0, \text{ other }x_i = 1\\
    \mathcal{L}_{4,2}(2):(1,1,1,1)(1,2,3,4)=A_3(B_2&+B_3+B_5+B_6)+A_2(B_2+B_5)+A_6(B_2+B_6)\\
    &x_2=x_4=1, \text{ other } x_i=0\\
    \mathcal{L}_{4,2}(3):(1,1,1,1)(1,2,3,4)=A_5(B_1&+B_3+B_5+B_7)+A_4(B_1+B_5)+A_6(B_1+B_7)\\
    &x_6=1, \text{ other } x_i = 0\\
    \mathcal{L}_{4,1}(1):(1,1,1,1)(1,2,3,4)=A_1(B_3&+B_4+B_6+B_7)+A_3(B_3+B_6)+A_5(B_3+B_7)\\
    &x_1=x_3=x_5=1, \text{ other } x_i = 0\\
    \mathcal{L}_{4,1}(2):(1,1,1,1)(1,2,3,4)=A_2(B_2&+B_4+B_5+B_7)+A_3(B_2+B_5)+A_6(B_2+B_7)\\
    &x_2=x_4=x_7=1, \text{ other } x_i = 0\\
    \mathcal{L}_{4,1}(3):(1,1,1,1)(1,2,3,4)=A_4(B_1&+B_4+B_5+B_6)+A_5(B_1+B_5)+A_6(B_1+B_6)\\
    &x_6=x_8=1, \text{ other } x_i = 0
\end{align*}

{}From the above we see that the identity of Theorem~\ref{VY} implies each of the column expansion identities for $m=4$.  In the other direction, the column expansion identities imply the identity of Theorem~\ref{VY} because
\begin{equation}\label{eq in terms of L}
\begin{split}
    &(1,1,1,1)(1,2,3,4)\\
    &=L_{4,3}(1) + L_{4,3}(2) - L_{4,3}(3) \\
    &\quad +y_1(L_{4,4}(1) - L_{4,2}(1)) + y_2(L_{4,4}(1)-L_{4,3}(1)) + y_3(L_{4,4}(2) - L_{4,3}(2)) + y_4(L_{4,4}(3)-L_{4,2}(2)) \\
    &\quad +y_5(L_{4,4}(2) - L_{4,1}(1)) + y_6(L_{4,3}(3) - L_{4,2}(3)) + y_7(L_{4,4}(3)-L_{4,1}(2)) + y_8(L_{4,3}(3) - L_{4,1}(3))
\end{split}     
\end{equation}
is the identity of Theorem~\ref{VY} where we have used the invertible change of variables $y_3 = 1-x_3 -x_2 + x_5$, $y_4 = 1-x_4-x_1+x_7$, $y_6=1-x_6-x_1+x_8$ and $y_i=1-x_i$ for $i=1,2,5,7,8$ to make it tidier.

In \cite{vlasev}, the proof that this was the most general identity of the form $(1,1,1,1)(1,2,3,4) = \sum \alpha_{i,j}A_iB_j$ and hence that there are 8 free variables in this identity (which is as it should be according to the conjecture), was done as follows.  Any identity true of all graphs is also true of particular large graphs.  For some particular large graphs the $A_i$ and $B_j$ were computed explicitly as was $(1,1,1,1)$ and $(1,2,3,4)$ and the general linear equation relating them was solved.  This is what first gave the identity of Theorem~\ref{VY}, and this argument shows that there can be no more than 8 free variables, though there could be fewer if some of the relations which are true on the particular large graphs are not true in general.  The next step in the proof of \cite{vlasev}, then, was to prove the identity from other known determinantal identities, showing that there were in fact no spurious identities from the particular large graphs and hence that Theorem~\ref{VY} holds for all graphs.

This proof does not readily generalize as the determinantal manipulations and the large explicit graphs used there were ad-hoc.  The arguments from the beginning of this section show that the column expansion identities imply the identity of Theorem~\ref{VY}, but further, the column expansion identities explain the number of free variables, as we will show in the remainder of this section.

The twelve $\mathcal{L}_{r,c}(\ell)$ identities for $m=4$ fall into four natural groupings based on the column removed from the Laplacian. Notice that when we add the right hand sides of the identities for each of these groupings (for example $L_{4,4}(1)+L_{4,4}(2)+L_{4,4}(3)$), we get the same sum, namely the one in which all eight of the variables $x_i$ are set to $1$. This gives us exactly one of every non-forbidden $A_iB_j$ monomial. 
Because we know that each of the four groupings of $L_{r,c}(\ell)$ are equal, we note that we can write each of the following three $L_{r,c}(\ell)$ in terms of the fourth grouping:        
\begin{align*}
L_{4,3}(3) & = L_{4,4}(1) + L_{4,4}(2) + L_{4,4}(3) - L_{4,3}(1) - L_{4,3}(2)\\   
L_{4,2}(3) & = L_{4,4}(1) + L_{4,4}(2) + L_{4,4}(3) - L_{4,2}(1) - L_{4,2}(2) \\
L_{4,1}(3) & = L_{4,4}(1) + L_{4,4}(2) + L_{4,4}(3) - L_{4,1}(1) - L_{4,1}(2)
\end{align*}

We now want to understand the dimension of $X_4$ and hence the number of free variables in Theorem~\ref{VY}.  To do so, we homogenize by subtracting $L_{4,1}(2)$ from the rest and using the three equations above to see that $L_{4,3}(3) - L_{4,1}(2)$, $L_{4,2}(3)-L_{4,1}(2)$ and $L_{4,1}(3)-L_{4,1}(2)$ can be written in terms of the others.  This leaves us with 

\begin{equation}\label{eq basis}
\begin{gathered}
L_{4,4}(1) - L_{4,4}(3), \qquad  L_{4,4}(2) - L_{4,4}(3), \\
L_{4,3}(1) - L_{4,4}(3), \qquad  L_{4,3}(2) - L_{4,4}(3), \qquad  L_{4,2}(1) - L_{4,4}(3) \\
L_{4,2}(2) - L_{4,4}(3), \qquad  L_{4,1}(1) - L_{4,4}(3), \qquad L_{4,1}(2) - L_{4,4}(3)
\end{gathered}
\end{equation}

We know from Theorem~\ref{VY} that the dimension is 8, so provided these eight differences of $L_{r,c}(\ell)$ are linearly independent in the vector space of linear combinations of $A_iB_j$ monomials, we will have shown that \eqref{eq basis} is a basis of $X_4$.  Here we can simply bootstrap this off of Theorem~\ref{VY} by noticing that the change of variables matrix between the eight differences of \eqref{eq basis} and the eight differences in \eqref{eq in terms of L} is
\[
\begin{bmatrix}
     1 & 0 & 0 & 0 & -1 & 0 & 0 & 0 \\ 
     1 & 0 & -1 & 0 & 0 & 0 & 0 & 0 \\
     0 & 1 & 0 & -1 & 0 & 0 & 0 & 0 \\
    0 & 0 & 1 & 0 & 0 & -1 & 0 & 0 \\
     0 & 1 & 0 & 0 & 0 & 0 & -1 & 0 \\
     0 & 0 & -1 & -1 & 1 & 1 & 0 & 0 \\
    0 & 0 & 0 & 0 & 0 & 0 & 0 & -1 \\
    0 & 0 & -1 & -1 & 0 & 0 & 1 & 1 
\end{bmatrix}
\]
which is non-singular, and hence \eqref{eq basis} gives an explicit basis of $X_4$.

\subsection{Quadratic spanning forest identities for general $m$}

We just established the set of identities for $m=4$ marked vertices. By applying the same technique to generalized column expansion identities, we can find the quadratic spanning forest identities for any $m$. Vlasev and the second author \cite{vlasev} conjectured that there would be $m(m-2)$ free variables in spanning forest identities with $m$ marked vertices (Conjecture \ref{VYConj} or rephrased as Conjecture~\ref{conj rephrased}).

As with the case when $m=4$, in order to get a quadratic spanning forest identity on $m$ marked vertices, we must apply one of $m-1$ column expansion identities to a Laplacian with a row and column already removed. As we saw in Lemma \ref{row}, the row removed does not impact the resultant identity, so we only need to decide which of the $m$ marked columns to remove from the Laplacian, and which of the $m-1$ column expansion identities to apply. This gives us $m(m-1)$ identities. However, as with the case when $m=4$, these identities will be interrelated; the rest of this section will discuss how they are related.

To determine how the identities are related, we must recall from Lemma \ref{exactmonomial} that every permissible monomial will show up exactly one time in the identity. We will also need a new definition.

\begin{definition}
A \textbf{block} of identities is the set of $m-1$ identities that all have the same column removed.
\end{definition}

Recall that the $A$ partition of an $AB$ monomial places two of the marked vertices into one part of the partition, and leaves the rest of the marked vertices as singletons, each in their own part. We will denote by $A_{\ell,j}$ the $A$ partition that pairs vertices $\ell$ and $j$ into one part and leaves the rest of the marked vertices as singletons.

Let $c$ be the column removed from the Laplacian in a block of identities, and let $j$ be a marked vertex. We will define $j'=j$ if $j<c$ and $j'=j-1$ if $j>c$.  In other words $j'$ indexes the same row or column that $j$ did before removing $c$. Recall from Lemma \ref{exactmonomial} that every permissible monomial (and only permissible monomials) appears exactly once in $L_{m,c}(j')$; that is $j$ will be part of the pair in the $A$ partition (let us call its partner vertex $\ell$), and $j$ will be in a different part from both $c$ and $\ell$ in the $B$ partition. With this notation, we can state the following lemma.

\begin{lemma}\label{AC}
Let $c$ be the column removed from the Laplacian in a block of identities. Then a monomial containing the partition $A_{c,\ell}$ will appear once in the identity $L_{m,c}(\ell')$, and will not appear in any other identity in the block. The $B$ partition in this monomial will have $c$ and $\ell$ in different parts.
\end{lemma}

\begin{proof}
We start with an arbitary monomial containing $A_{c,\ell}$. By definition of permissible monomial, the $B$ partition of this monomial must have $c$ and $\ell$ in separate parts. We first show that it is possible for that monomial to appear in $L_{m,c}(\ell')$. If we look at the column expansion identity for $L_{m,c}(\ell')$, we will get, on the right hand side, $\det(L_{\{i,k+1\},\{\ell,c\}})$ for the $B$ part of the monomial. This monomial requires that $\ell$ and $c$ be in separate parts, so an accompanying $A$ that has $\ell$ and $c$ in the same part will be non-forbidden. By Lemma \ref{exactmonomial}, since it is possible for any monomial containing $A_{c,\ell}$ to show up in $L_{m,c}(\ell')$, then every monomial containing $A_{c,\ell}$ will show up exactly once in this identity. 

On the other hand, note that any monomial containing $A_{c,\ell}$ will not show up in any other identity in the block aside from $L_{m,c}(\ell')$. Suppose $A_{c,\ell}$ appears in the identity $L_{m,c}(j')$. Then $j$ must be part of the pair of $A$. Since $j$ cannot equal $c$ by definition of $L_{m,c}(j')$, then the only way for $A$ to have $c$ and $\ell$ in the same partition is if $\ell=j$.
\end{proof}

\begin{theorem}\label{monomial}
Every identity block sums to an identity with $m-1$ copies of $(1,1,\dots,1)(1,2,\dots,m)$ on the left hand side and exactly one copy of each permissible monomial on the right hand side.
\end{theorem}

\begin{proof}
Since every quadratic spanning forest identity has $(1,1,\dots,1)(1,2,\dots,m)$ on the left hand side, if we sum all $m-1$ identities in a block, there will be $m-1$ copies of that on the left hand side. Therefore our left hand side is as expected, and we need only focus on the right hand side of the identity.  Fix an arbitrary block and let $c$ be the removed column.

By Lemma \ref{AC}, if a monomial contains $A_{c,\ell}$, then it appears exactly in $L_{m,c}(\ell')$. Thus any monomial containing $A_{c,\ell}$ appears exactly once within the block.

Now suppose a monomial contains $A_{\ell,j}$ where $\ell,j\neq c$. By Lemma \ref{exactmonomial}, this monomial could only be contained in either $L_{m,c}(j')$ or $L_{m,c}(\ell')$. The identity that the monomial appears in will be determined by the accompanying $B$. The $B$ paired with $A_{\ell,j}$ must have $\ell$ and $j$ in separate partitions to be non-forbidden. Thus $c$ will either be in a partition with $\ell$ or in a partition with $j$ (since all $B$s partition the vertices into two groups). In $L_{m,c}(j')$, we get $\det(L_{\{i,m\},\{j,c\}})$ for the $B$ part of the monomial. This puts $j$ in a different partition than $c$. Thus if $B$ puts $c$ in the same partition as $j$, it cannot show up in $L_{m,c}(j')$. It can, (and by Lemma \ref{exactmonomial} must), however, show up in $L_{m,c}(\ell')$ since that identity requires $\ell$ to be in the pair in $A$ (and the rest singletons), so $B$ (which pairs $c,j\neq \ell$) will make the pair permissible. Thus if our monomial contains $A_{\ell,j}$ and a $B$ that pairs $c$ with $j$, the monomial will show up exactly once in $L_{m,c}(\ell')$. Similarly, if our monomial contains $A_{\ell,j}$ and a $B$ that pairs $c$ with $\ell$, the monomial will show up exactly once in $L_{m,c}(j')$.

We have shown that every possible $AB$ monomial shows up in exactly one of the identities within a given block. Thus every block of identities must sum to give exactly one copy of every non-forbidden monomial on the right hand side.
\end{proof}

We have shown that the $m(m-1)$ identities are interrelated in blocks of size $m-1$. Now let's bring our attention to Conjecture~\ref{conj rephrased} itself.  Each of the $m(m-1)$ identities is an expression of the form $(1,1,\ldots,1)(1,2,\ldots, m) = \sum_{i,j}\alpha_{i,j}A_iB_j$, so the furthermore of the conjecture holds.  As in $m=3$ and $m=4$ we have candidates for a basis for $X_m$, namely fix one of the the $L_{m, c}(j)$, say $L_{m,m}(m-1)$, and subtract it from the others to obtain $m(m-1)-1$ elements of $X_m$.  Fix one of the blocks, say the $c=m$ block.  Subtracting the sum of $L$s in any other block from the sum of $L$s in the fixed block, we get $m-1$ identities of $L$s.  In each of the identities as many $L$s with positive signs as negative signs occur, so these identities can be rewritten in terms of differences $L_{m,i}(j) - L_{m,m}(m-1)$.   Now, we can solve for one of these differences in each block $1\leq i \leq m-1$, say, solve for $L_{m, i}(m-1)-L_{m,m}(m-1)$. Then remove these $m-1$ differences that have been solved for from our set of $m(m-1)-1$ elements of $X_m$.  There are $m(m-2)$ elements remaining in the set and these we claim form a basis for $X_m$.

The first thing to show is that the elements of the purported basis are linearly independent as elements of $X_m$.   

\begin{proposition}\label{prop lin ind}
The elements $\{L_{m, i}(j)-L_{m,m}(m-1)\}_{\substack{1\leq i \leq m\\ 1\leq j \leq m-2}}$ are linearly independent in $X_m$
\end{proposition}

\begin{proof}
Recall from the definitions in the paragraph before Conjecture~\ref{conj rephrased}, that $V_m$ is the vector space generated by monomials $A_iB_j$, and $X_m$ is a subspace of $V_m$, specifically the subset consisting of only those sums of $AB$s which are $0$ on all graphs.  Therefore, to show a set of elements of $X_m$ is linearly independent in $X_m$ it suffices to show that the set of elements is linearly independent in $V_m$.

Furthermore, let $J$ be the sum of all permissible $AB$ pairs.  Note that by Theorem~\ref{monomial} every block sums to $J$, that is;  $J = \sum_{j}L_{m,c}(j)$ for each column $c$.  It will be convenient to work in $V_m$ modulo the ideal generated by $J$.  We will first prove that there is no nontrivial identity $I$ of the $L_{m,i}(j)-L_{m,m}(m-1)$ in $V_m/\langle J \rangle$. 

Suppose for a contradiction that we have a nontrivial identity $I$ of the $L_{m,i}(j)-L_{m,m}(m-1)$ in $V_m/\langle J \rangle$.  Since the blocks sum to $J$, we can add any multiple of $\sum_j L_{m,c}(j)$ for any $c$ to $I$ without changing the identity in $V_m/\langle J \rangle$. Using the fact that $\sum_jL_{m,c}(j)=J$ which is the sum of all permissible $AB$ pairs each with coefficient $1$, we can add copies of $\sum_jL_{m,c}(j)$ to $I$ so as to shift the coefficients so that the coefficients of the $L_{m,c}(j)$ in $I$ are all nonnegative. In particular, because $I$ is nontrivial modulo blocks, we can ensure that every block has at least one $L$, say $L_{m,c_1}(j_1)$, that has a positive coefficient, and at least one $L$, say $L_{m,c_1}(\ell_1)$, that has a coefficient of zero.

Consider the monomials in $I$ that have $A_{j_1,\ell_1}$ as their $A$ partition. By Lemma \ref{exactmonomial}, within the $c_1$ block, these monomials would come from $L_{m,c_1}(j_1)$ and $L_{m,c_1}(\ell_1)$. However, because the coefficient on $L_{m,c_1}(\ell_1)$ is $0$, the monomials containing $A_{j_1,\ell_1}$ that would ordinarily have come from that identity must instead have come from another source. Specifically, we are looking at monomials with the $A$ partition being $A_{j_1,\ell_1}$ and the $B$ partition having $\ell_1$ in one part and $j_1$ and $c_1$ in the other part. We will show that no other source can recover all of the monomials containing $A_{j_1,\ell_1}$ that would have come from $L_{m,c_1}(\ell_1)$, and as a result $I$ cannot equal $0$ modulo the blocks, resulting in a contradiction.

There are three possible sources for the monomials containing $A_{j_1,\ell_1}$ in $L_{m,c_1}(\ell_1)$. One source is they may be recovered from the same $\ell_1$ expansion identity in a different block, that is to say from $L_{m,c_2}(\ell_1)$ for some $c_2\neq c_1$. Monomials from here would have the appropriate $A$ partition, and would have a $B$ partition that has $\ell_1$ in one part and $j_1c_2$ in the other part. The ones of these which we would want to replace missing $L_{m,c_1}(\ell_1)$ terms are those where $\ell_1$ is one one side and $j_1c_2c_1$ is on the other side (that is we get half the terms we would want). However, we also get terms where $j_1\ell_1$ are the pair in $A$ and in $B$ $\ell_1c_1$ is in one part and $j_1c_2$ is in the other part.  These terms all appear in $L_{m,c_1}(j_1)$ and we have the same number of these as of the terms we wanted, so  for every missing term of $L_{m,c_1}(\ell_1)$ we could pick up in this way we also pick up an additional term that we already had from $L_{m,c_1}(j_1)$. Thus this source cannot recover our missing monomials.

The second source of our monomials containing $A_{\ell_1,j_1}$ is from $L_{m,c_2}(j_1)$. Then the corresponding $B$ partition must have $j_1$ in one part and $\ell_1c_2$ in the other part.
Thus the ones that would replace the $L_{m,c_1}(\ell_1)$ terms are the ones with $j_1c_1$ in one part and $\ell_1c_2$ in the other part. However, similar to the previous paragraph, we also get terms where in the $B$ partition $\ell_1c_2c_1$ appear in one part and $j_1$ in the other part. Again, this gives us extra terms from $L_{m,c_1}(j_1)$, so we cannot recover our missing monomials from here.

Thirdly, the monomials containing $A_{j_1,\ell_1}$ can also come from $L_{m,j_1}(\ell_1)$ or $L_{m, \ell_1}(j_1)$. In both cases all $B$s with $j_1$ and $\ell_1$ separated will be permissible with $A_{j_1,\ell_1}$. Then we will get an equal number of those with $j_1c_1$ in one part and $\ell_1$ in the other as we get with $j_1$ in one part and $\ell_1c_1$ in the other. Thus, again, for each missing monomial we recover, we add another monomial to the total in $L_{m,c_1}(j_1)$. Thus we cannot completely recover the monomials missing from the $0$-coefficient $L_{m,c_1}(\ell_1)$. As a result, $I$ cannot equal zero, so it cannot be a nontrivial identity in $V_m/\langle J\rangle$.

Finally, note that in the argument above we only used $J$ in the form of block sums $\sum_jL_{m,c}(j)$, so in fact we have proved that if there is a nontrivial identity $I$ of elements of the purported basis then it not only must be a multiple of $J$ in $V_m$, but further it must be a linear combination of block sums $\sum_jL_{m,c}(j)$ in $V_m$.  However, the block sums are exactly the identities that were removed in the construction of the purported basis, so this is impossible by construction.  Therefore the purported basis is linearly independent in $V_m$ and hence in $X_m$.
\end{proof} 

\subsection{Obtaining all identities}

The final part is to show that the purported basis given above spans $X_m$.  In the $m=3$ and $m=4$ cases this was simply a direct computation, but for the general case, we must approach it differently, and will do so by induction.

It will be convenient for the induction to work modulo $(1,1,\ldots)(1,2,\ldots)$.  Then the statement we need to prove becomes that every sum of $AB$ pairs which is 0 modulo $(1,1,\ldots)(1,2,\ldots)$ can be written as a sum of $L_{r,c}(j)$.  This is the content of Theorem~\ref{finaltheorem}. 

Let us first establish some notation. Recall that $A_{\ell,k}$ is a partition on marked vertices such that $\ell$ and $k$ are in the same part and every other vertex is in a part by itself. Similarly, we will write $A_{\neq i}$ to mean a partition $A$ where $i$ is not in the part with two vertices (i.e. $i$ is in a part by itself) and $A_{\ell, \neq i}$ to mean a partition $A$ where $i$ is not in the part with two vertices but $\ell$ is (the other member of the part with $\ell$ being unspecified).  These latter two notations will appear in the context of sums where we will be summing over partitions with these constraints. Let $B^{P_1,P_2}$ be a partition on marked vertices with two parts: $P_1$ and $P_2$. Let $B^{P_1i,P_2}$ be our partition $B$ where vertex $i$ is specifically in part $P_1$ and let $B^{P_1,P_2i}$ be the same $B$ partition but with $i$ in the other part.

\begin{theorem} \label{finaltheorem}
For $n$ marked vertices, any identity of the form $$\eta (1,1,\dots,1)(1,2,\dots,m)=\sum \alpha_{\ell,k,P_1,P_2}A_{\ell,k}B^{P_1,P_2},$$ where $\eta$ and $\alpha_{\ell,k,P_1,P_2}$ are arbitrary coefficients, can be written as $$\eta (1,1,\dots,1)(1,2,\dots,m)=\sum \varepsilon_{r,c,j}L_{r,c}(j),$$ .
\end{theorem}

We will prove this theorem at the end of this section, but first we need some lemmas to make the proof go more smoothly.

We are going to prove Theorem \ref{finaltheorem} using induction on the number of marked vertices.  The base case is given in Section~\ref{subsec 4 vert}, so for the remaining parts of this section, we will assume our inductive hypothesis, namely that Theorem \ref{finaltheorem} holds for $m-1$ marked vertices.

\begin{lemma}\label{firstdifference}
Assume that Theorem \ref{finaltheorem} holds for $m-1$ marked vertices. For $m$ marked vertices (one of which is $i$), any identity of the form $$\eta (1,1,\dots,1)(1,2,\dots,m)=\sum \alpha_{\ell,k,P_1,P_2}A_{\ell,k}B^{P_1,P_2},$$ where $\eta$ and $\alpha_{\ell,k,P_1,P_2}$ are arbitrary coefficients, can be written as $$\eta (1,1,\dots,1)(1,2,\dots,m)=\frac{1}{2}\sum_{c,j} \varepsilon_{i,c,j}L_{i,c}(j)+\sum_{k,P_1,P_2} \alpha_{i,k,P_1,P_2}A_{i,k}B^{P_1,P_2}+\sum \beta_{A_{\neq i}, P_3,P_4}(A_{\neq i}(B^{P_3i,P_4}-B^{P_3,P_4i})),$$
where the last sum runs over all $A_{\neq i}$ as well as over $P_3$ and $P_4$ and the coefficient $\beta_{A_{\neq i}, P_3, P_4}$ may depend on the specific $A$ partition as well as on $P_3$ and $P_4$.
\end{lemma}

\begin{proof}
Let us begin with a graph $G$ on $m$ marked vertices, and let \begin{equation}\label{eq eta sum}\sum \alpha_{\ell,k,P_1,P_2}A_{m,k}B^{P_1,P_2}\end{equation} be a sum on $G$ that sums to $\eta(1,1,\dots,1)(1,2,\dots,m)$. Let $i$ be one of the $m$ marked vertices. 

We want to understand the behaviour of the polynomials when vertex $i$ is removed.  In order to do this, scale every edge coming out of $i$ with a new indeterminant $t$. We want to look at the part of the sum that is linear with respect to $t$. If $A$ does not have $i$ in its pair ($A_{\neq i}$), then to be linear in $t$, we have the part of this $AB$ where $i$ is a leaf in the spanning forests for $B$ and $i$ is an isolated vertex in the spanning forests for $A$. If $A$ does have $i$ as a paired vertex, then $t$ will show up in $A$, so in order to be linear it cannot show up in $B$. Thus in the part of this $AB$ which is linear in $t$, $i$ will be an isolated vertex in $B$, and the other part of $B$ will contain the rest of the marked vertices, while $i$ will be a leaf in $A$.  

These partitions and their corresponding spanning forests and polynomials can all be naturally interpreted on the graph with $i$ removed.  Let $H_i$, a graph on $m-1$ marked vertices, be $G$ with $i$ removed. In the part of the sum \eqref{eq eta sum} where $t$ is linear, any $AB$ pair where $A=A_{i,-}$ (i.e. $i$ is a paired vertex in $A$) will become a copy of $(1,1,\dots,1)(1,2,\dots m-1)$ in $H_i$: with the $i$ removed, every special vertex from $A_{i,-}$ will now be in a separate part, giving us $(1,2,\dots,m-1)$; and in order for $t$ to be linear, $B$ must have $P_1=i$ and $P_2$ the rest of the vertices, giving us $(1,1,\dots,1)$ in $H_i$. Also, making the same $t$ scaling on $\eta(1,1,\ldots, 1)(1,2,\ldots,m)$ and taking the linear part in $t$, we see that $i$ must be isolated in $(1,1,\ldots, 1)$ and a leaf in $(1,2,\ldots, m)$, so reinterpreting on $H_i$ we get $(1,1,\ldots, 1)(1,2\ldots, m-1)$.

This means that the linear part in $t$ of $\eqref{eq eta sum}$ can be interpreted as a sum of $AB$ on $m-1$ marked vertices that equals to copies of $(1,1,\dots,1)(1,2,\dots m-1)$. 
We can apply Theorem \ref{finaltheorem} to get this sum on $H_i$ written as a sum of $L$ identities: 
$\sum \varepsilon_{i,c,j} L^{H_i}_{i,c}(j),$ 
where $i$ is used as the row removed simply as a notational device in order to keep track of the subgraph that this identity originally came from, and the superscript $H_i$ is used to indicate that we are looking at the identity $L_{i,c}(j)$ restricted to the subgraph $H_i$.

Notice that we can lift this sum $\sum \varepsilon_{i,c,j}L^{H_i}_{i,c}(j)$ in $H_i$ to the same sum in $G$ by summing over the same $i,c,j$ to get $\sum \varepsilon_{i,c,j} L^G_{i,c}(j)$. This lifting gives us some extra $A_{i,-}B$ terms that were not in the original $L^{H_i}$'s, and it gives us twice as many $A_{\neq i}B$ parts since $i$ can belong to either $P_1$ or $P_2$ in $B$.

Let us take our original sum of $AB$ pairs and subtract off half this lifted $L$ sum (we take half since there are twice as many $A_{\neq i}B$ parts as in $H_i$): $$\sum \alpha_{\ell,k,P_1,P_2}A_{\ell,k}B^{P_1,P_2}-\frac{1}{2}\sum \varepsilon_{i,c,j} L^G_{i,c}(j).$$

To see what this evaluates to, let us look at a single set of pairs in our original sum: Take some $A_{\neq i}$.  In our original sum we have $\alpha A_{\neq i}B^{P_3i,P_4}+\beta A_{\neq i}B^{P_3,P_4i}$ for this particular $A_{\neq i}$ and for some coefficients $\alpha$ and $\beta$. When we collapse down to $H_i$, both $A_{\neq i}B^{P_3i, P_4}$ and $A_{\neq i}B^{P_3, P_4i}$ become the same $AB$ pair in $H_i$. Let us say the coefficient on that $AB$ pair in $H_i$ becomes $\varepsilon$ (which might be a sum of different $\varepsilon_{i,c,j}$'s on $L_{i,c}(j)$). Since this has to be the same number as the number of times this pair appears in $H_i$, we know that $\alpha+\beta=\varepsilon$. Then when we lift this sum of $L$'s back to $G$ and subtract half, we get 
\begin{align*}
\alpha A_{\neq i}B^{P_3i,P_4}&+\beta A_{\neq i}B^{P_3,P_4i}-\frac{\varepsilon}{2} A_{\neq i}B^{P_3i,P_4}-\frac{\varepsilon}{2} A_{\neq i}B^{P_3,P_4i}\\
&=\left(\alpha-\frac{\alpha+\beta}{2} \right)A_{\neq i}B^{P_3i,P_4}+\left(\beta-\frac{\alpha+\beta}{2}\right) A_{\neq i}B^{P_3,P_4i}\\
&=\left(\frac{\alpha}{2}-\frac{\beta}{2} \right)A_{\neq i}B^{P_3i,P_4}+\left(\frac{\beta}{2}-\frac{\alpha}{2}\right) A_{\neq i}B^{P_3,P_4i}\\
&=\left(\frac{\alpha}{2}-\frac{\beta}{2} \right)(A_{\neq i}B^{P_3i,P_4}-A_{\neq i}B^{P_3,P_4i})
\end{align*}
for this particular $A_{\neq i}$.

This means that, using this obesrvation now on all the $A_{\neq i}$ when we take our original sum and subtract our lifted sum, we get $$\sum \alpha_{\ell,k,P_1,P_2}A_{\ell,k}B^{P_1,P_2}-\frac{1}{2}\sum \varepsilon_{i,c,j} L^G_{i,c}(j)=\sum \alpha_{i,k,P_1,P_2}A_{i,k}B^{P_1,P_2}+\sum \beta_{A_{\neq i},P_3,P_4}(A_{\neq i}(B^{P_3i,P_4}-B^{P_3,P_4i})).$$

By our original assumption, $\sum \alpha_{\ell,k,P_1,P_2}A_{\ell,k}B^{P_1,P_2}=\eta(1,1,\dots,1)(1,2,\dots,m)$. Substituting this in and adding $\frac{1}{2}\sum \varepsilon_{i,c,j} L^G_{i,c}(j)$ to both sides yields the result.

\end{proof}

\begin{lemma}\label{ispecial}
Let $H_x$ be $G$ with marked vertex $x$ removed.
Consider a difference $A_{\neq i}(B^{P_3i,P_4}-B^{P_3,P_4i})$ in $H_x$ for any $x\neq i$ that appears as part of an $AB$ sum which has been expressed in terms of $L$s.  Any $L$s contributing to this difference must be of the form $L_{x,i}(\ell)$.
\end{lemma}

\begin{proof}
There are three options for the $L$s that could exist in $H_x$: $L_{x,i}(\ell)$, $L_{x,c}(i)$, or $L_{x,c}(\ell)$ where $c,\ell\neq i$. We will show that $L_{x,c}(i)$ and $L_{x,c}(\ell)$ cannot contribute to the difference. Let us start with $L_{x,c}(i)$: by our combinatorial interpretation of the determinants in the column expansion identity, this would require $i$ to be in the part of size 2 in $A$ for any monomial that appears in the identity. However, our difference specifically requires that $i$ is not in the part of size 2 in $A$, so $L_{x,c}(i)$ cannot contribute to the difference.

Now if we have $L_{x,c}(\ell)$ with $c,\ell\neq i$, then there is no restriction inherent in the identity on what part of each partition $i$ can belong to. Since in the difference $i$ is not in the part of size 2 in $A$, then by Lemma~\ref{exactmonomial} $i$ can belong to either of the parts in $B$. However, any monomial that can show up according to Lemma~\ref{exactmonomial} will show up, which means that if $L_{x,c}(\ell)$ contains $A_{\neq i}B^{P_3i,P_4}$, it will also contain $A_{\neq i}B^{P_3,P_4i}$ with the same coefficient. Thus this cannot contribute to a difference of the two monomials.

Because $L_{x,c}(\ell)$ and $L_{x,c}(i)$ cannot contribute to the difference, that leaves only $L_{x,i}(\ell)$ to contribute.
\end{proof}

We are almost ready to prove Theorem \ref{finaltheorem}, but before we do, let us re-examine 
\begin{equation}\label{eq re examine}
\sum \alpha_{\ell,k,P_1,P_2}A_{\ell,k}B^{P_1,P_2}-\frac{1}{2}\sum \varepsilon_{i,c,j} L^G_{i,c}(j)=\sum \alpha_{i,k,P_1,P_2}A_{i,k}B^{P_1,P_2}+\sum \beta_{P_3,P_4}(A_{\neq i}B^{P_3i,P_4}-A_{\neq i}B^{P_3,P_4i}).
\end{equation}
Specifically, let us compare differences in pairs from the second term on the right hand side with a special focus on where a second marked vertex $x$ belongs in $B$: $\varepsilon_{P_3}(A'B^{P_3ix,P_4}-A'B^{P_3x,P_4i})$ and $\varepsilon_{P_4}(A'B^{P_3i,P_4x}-A'B^{P_3,P_4ix})$ where $A'$ is a fixed partition $A$ such that neither $i$ nor $x$ are special in $A$.  Note that it is not necessarily true that $\varepsilon_{P_3}$ and $\varepsilon_{P_4}$ equal each other. Let us consider what happens when we now bring \eqref{eq re examine} down to the graph $H_x$ with marked vertex $x$ removed using the same technique of scaling all the edges incident to $x$ and then considering the linear part.  As before, the resulting expression in $H_x$ will be a sum of $AB$ pairs equal to $(1,1,\ldots)(1,2,\ldots, m-1)$ and hence by the inductive hypothesis will be expressible in terms of $Ls$ on $H_x$.  When we remove the vertex $x$ from $G$ to bring \eqref{eq re examine} down to $H_x$, $\varepsilon_{P_3}A'B^{P_3ix,P_4}$ and $\varepsilon_{P_4}A'B^{P_3i,P_4x}$ both collapse to the same monomial, which either appears or does not in each of the $L$s in $H_x$ which give the sum. Then, as before, raise the sum of $Ls$ up to $G$ and subtract it from \eqref{eq re examine} (with a coefficient $1/2$), writing the $Ls$ as $Ls$ on the left and as their sums of $AB$ pairs on the right.  Then on the right hand side the two differences we focused on become $$\left(\frac{\varepsilon_{P_3}}{2}-\frac{\varepsilon_{P_4}}{2}\right)(A'B^{P_3ix,P_4}-A'B^{P_3i,P_4x})$$

Similarly, $-\varepsilon_{P_3}A'B^{P_3x,P_4i}$ and $-\varepsilon_{P_4}A'B^{P_3,P_4ix}$ both collapse to the same monomial in $H_x$, so when we raise it back up to $G$ and subtract we will get $$-\left(\frac{\varepsilon_{P_3}}{2}-\frac{\varepsilon_{P_4}}{2}\right)(A'B^{P_3x,P_4i}-A'B^{P_3,P_4xi})$$ 

Let $\varepsilon=\left(\frac{\varepsilon_{P_3}}{2}-\frac{\varepsilon_{P_4}}{2}\right)$. Then taking these together we have 

\begin{equation}\label{combinedmonomials}
\varepsilon(A'B^{P_3ix,P_4}+A'B^{P_3,P_4ix}-A'B^{P_3i,P_4x}-A'B^{P_3x,P_4i})
\end{equation}

\begin{lemma}\label{zerocoeff}
With set up and notation as above, in Equation \ref{combinedmonomials}, $\varepsilon=0$.
\end{lemma}

\begin{proof}
Let $y$ be a third vertex that is also not in the part of size 2 in $A'$. 
Notice that we can collect the terms in \eqref{combinedmonomials} in two different ways:
$\varepsilon(A'(B^{P_3ix, P_4} - B^{P_3x, P_4i}) - A'(B^{P_3i, P_4x}- B^{P_3, P_4ix}))$ and $\varepsilon(A'(B^{P_3ix, P_4} - B^{P_3i, P_4x}) - A'(B^{P_3x, P_4i}- B^{P_3, P_4ix}))$.
By Lemma \ref{ispecial}, any $L$ contributing to the first of these must be of the form $L_{y,i}(\ell)$ and any $L$ contributing to the second must be of the form $L_{y,x}(\ell)$. However these are the same expression and so any $L$ contributing here must have both column $i$ and column $x$ removed, which is imposible.  Thus $\varepsilon=0$.
\end{proof}

We now have all the pieces we need to prove Theorem \ref{finaltheorem}.

\begin{proof} \textbf{(Theorem \ref{finaltheorem})}. We will proceed using induction on the number of marked vertices. As our base case, we will use $m=4$; see section 3.1. For our inductive hypothesis, we assume that Theorem \ref{finaltheorem} is true for $m-1$ marked vertices.

Let us begin with a graph $G$ on $m$ marked vertices, and let $\sum \alpha_{\ell,k,P_1,P_2}A_{\ell,k}B^{P_1,P_2}$ be a sum on $G$ that sums to $\eta(1,1,\dots,1)(1,2,\dots,m)$. Let $i$ be one of the $m$ marked vertices. By Lemma \ref{firstdifference}, we can subtract off copies of $\varepsilon_{i,c,j} L^G_{i,c}(j)$ to get: $$\sum \alpha_{\ell,k,P_1,P_2}A_{\ell,k}B^{P_1,P_2}-\frac{1}{2}\sum \varepsilon_{i,c,j} L^G_{i,c}(j)=\sum \alpha_{i,k,P_1,P_2}A_{i,k}B^{P_1,P_2}+\sum \beta_{A_{\neq i}, P_3,P_4}(A_{\neq i}(B^{P_3i,P_4}-B^{P_3,P_4i})).$$

We can apply a similar process as seen in Lemma \ref{firstdifference} by removing another vertex from $G$, let's say $x$. Then again, we will have a subgraph $H_x$ on $m-1$ marked vertices, so we can use the inductive hypothesis again:

\begin{align*}
    \sum \alpha_{\ell,k,P_1,P_2}A_{\ell,k}B^{P_1,P_2}-\frac{1}{2}\sum \varepsilon_{i,c,j} L^G_{i,c}(j)-\frac{1}{2}\sum \varepsilon_{x,c,j} L^G_{x,c}(j)
    =&\sum \alpha_{i,x,P_1,P_2}A_{i,x}B^{P_1,P_2}\\
    &+\sum \beta_{A_{x,\neq i}P_3,P_4}(A_{x,\neq i}(B^{P_3i,P_4}-B^{P_3,P_4i}))\\
    &+\sum \gamma_{A_{\neq x},P_3,P_4}(A_{\neq x}(B^{P_3x,P_4}-B^{P_3,P_4x}))
\end{align*}
where as earlier in this section, the coefficients involving $A_{x,\neq i}$ and $A_{\neq i}$ indices depend on the particular partition of the $A$ for that term and the sums run over all $A$s with the indicated restrictions on which vertices participate in their part of size 2 as well as over $P_3$, $P_4$.

We will do this one more time, picking a third vertex to remove from $G$, let's call it $y$. Then again we have a subgraph on $H_y$ on $m-1$ marked vertices, so again we use the inductive hypothesis.

\begin{align*}
    \sum \alpha_{\ell,k,P_1,P_2}A_{\ell,k}B^{P_1,P_2}-\frac{1}{2}\sum \varepsilon_{i,c,j} L^G_{i,c}(j)&-\frac{1}{2}\sum \varepsilon_{x,c,j} L^G_{x,c}(j)-\frac{1}{2}\sum \varepsilon_{y,c,j} L^G_{y,c}(j)\\
    =&\sum \beta_{x,y,P_3,P_4}(A_{x,y}(B^{P_3i,P_4}-B^{P_3,P_4i}))\\
    &+\sum \gamma_{A_{y, \neq x},P_3,P_4}(A_{y,\neq x}(B^{P_3x,P_4}-B^{P_3,P_4x}))\\
        &+\sum \delta_{A_{\neq y},P_3,P_4}(A_{\neq y}(B^{P_3y,P_4}-B^{P_3,P_4y}))
\end{align*}
with the dependence of the coefficients and ranges of the sums as described above.

This tells us that, with the exception of differences of pairs that have the same $A$ and a $B$ that is only different by which partition one vertex is in, our original sum of $AB$s can be expressed as sums of our $L$ identities. 

It only remains to show that the coefficients of these differences are zero. In Lemma \ref{zerocoeff}, we chose an arbitrary difference and showed that the coefficient of that difference was zero. The only assumption we made for Lemma \ref{zerocoeff} is that there are at least $5$ marked vertices in $G$. Because our base case has $4$ marked vertices, we can safely assume the inductive hypothesis for a $G$ of $5$ marked vertices, so the lemma holds. As a result, we can update our last equation to read: $$\sum \alpha_{\ell,k,P_1,P_2}A_{\ell,k}B^{P_1,P_2}-\frac{1}{2}\sum \varepsilon_{i,c,j} L^G_{i,c}(j)-\frac{1}{2}\sum \varepsilon_{x,c,j} L^G_{x,c}(j)-\frac{1}{2}\sum \varepsilon_{y,c,j} L^G_{y,c}(j)=0.$$ Moving the $L$s over to the right hand side, we have successfully written our arbitrary identity as a linear combination of the column expansion identity-derived quadratic spanning forest identities.

\end{proof}

Note that we did not assume that the $AB$ pairs were non-forbidden pairs, but since $L$s only involve non-forbidden pairs, we get as a consequence that any sum of $AB$s which equals a multiple of $(1,1,\ldots)(1,2,\ldots, m)$ consists of only non-forbidden pairs.

\subsection{Proof of the conjecture}

Now we are ready to prove the conjecture.

\begin{theorem}
(Conjecture \ref{conj rephrased}).  The dimension of $X_m$ is $m(m-2)$ and there is at least one expression of the form $(1,1,\ldots, 1)(1,2,\ldots, m) = \sum_{i,j}\alpha_{i,j}A_iB_j$.
Consequently, (Conjecture \ref{VYConj}), the formulae for quadratic spanning forest identities of the $AB$ type on $m$ marked vertices have $m(m-2)$ free variables.  
\end{theorem}

\begin{proof}
Each $\mathcal{L}$ is an expression of the form $(1,1,\ldots, 1)(1,2,\ldots, m) = \sum_{i,j}\alpha_{i,j}A_iB_j$.  Proposition~\ref{prop lin ind} gives a linearly independent set in $X_m$ of size $m(m-2)$.

We need only to show that this linearly independent set spans $X_m$.  From Theorem \ref{finaltheorem} we know that any sum of $AB$ pairs which is equal to a multiple of $(1,1,\ldots,1)(1,2,\ldots,m)$ can be expressed as a sum of $L$s.
Consider now an element $x\in X_m$.  $x$ is a sum of $AB$ pairs which is 0 on every graph.  In particular this sum is 0 modulo $(1,1,\ldots)(1,2,\ldots)$ and so by Theorem \ref{finaltheorem} can be written as a sum of $L_{r,c}(j)$.  If this sum of $L_{r,c}(j)$ does not have equal number of $L$s with positive and negative signs then using each $\mathcal{L}_{r,c}(j)$ we would get that this sum is a nonzero multiple of $(1,1,\ldots,1)(1,2,\ldots,m)$, but this sum is also $0$ by the original assumption on $x$.  Thus we would have that a nonzero multiple of $(1,1,\ldots, 1)(1,2,\ldots, m)$ is $0$ on all graphs, which is a contradiction, so the sum of the coefficients of the $L$s must be $0$.  
Therefore each $L_{r,c}(j)$ can be replaced by $L_{r,c}(j) - L_{m, m}(m-1)$ and so $x$ has been expressed in terms of the set from Proposition~\ref{prop lin ind}.  Since this holds for any $x\in X_m$, the set is a basis and the theorem is proved.
\end{proof}

\section{Conclusion}

Inspired by quantum field theory calculations, we were particularly interested in quadratic spanning forest polynomials identities.  The column expansion identities of determinants give a source of quadratic spanning forest polynomial identities and allowed us to prove a conjecture of \cite{vlasev}.  In brief, the $m(m-1)$ column expansion identities come in $m$ blocks each with the same sum, giving $m-1$ internal relations between them and homogenizing uses one further identity, leaving the expected dimension of $m(m-2)$ for the spanning forest identities of $AB$ form.
Furthermore, we give a combinatorial interpretation of all such spanning forest identities via an edge-swapping argument developed by the first author in \cite{dennis}.

We did not consider spanning forest identities which were quadratic but where the partitions involved had other numbers of parts.  These would also be useful in quantum field theory calculations and would be a good subject for future investigations.

\bibliographystyle{plain}
\bibliography{refs}

\end{document}